\documentclass[11pt]{amsart}
\usepackage{amsmath}
\usepackage{amsthm}
\usepackage{amsbsy}
\usepackage{amsopn}
\usepackage{amsmath,amscd}
\usepackage{amssymb}
\usepackage{amsfonts}
\usepackage{multirow}
\usepackage{mathrsfs} 
\usepackage[official ]{eurosym}

\newcommand\isomto{\stackrel{\sim}{\smash{\longrightarrow}\rule{0pt}{0.4ex}}}

\newcommand{\RNum}[1]{\uppercase\expandafter{\romannumeral #1\relax}}

\newcommand{\GL}{\textup{GL}}

 \usepackage[all]{xy}
                \usepackage{amsthm,amsfonts,amsmath,amscd,amssymb,epsfig,verbatim,
}

\usepackage{graphicx}
\usepackage{epstopdf}
\date{\today} 
 
\usepackage{graphicx}

\usepackage[dvipsnames,svgnames,x11names,hyperref]{xcolor}
\usepackage{url,graphicx,verbatim,amssymb,enumerate,stmaryrd}
\usepackage[pagebackref,colorlinks,citecolor=Bittersweet,linkcolor=Bittersweet,urlcolor=Bittersweet,filecolor=Bittersweet]{hyperref}
\usepackage{tikz} 
\usetikzlibrary{arrows,decorations.pathmorphing,backgrounds,fit,positioning,shapes.symbols,chains,calc}
\usepackage[all]{xy}
\xyoption{matrix}
\xyoption{arrow}
\usepackage[hmargin=3cm,vmargin=3cm]{geometry}
\usepackage{setspace,kantlipsum}
\tikzset{help lines/.style={step=#1cm,very thin, color=gray},
help lines/.default=.5} 

\usepackage{booktabs}

\theoremstyle{plain}
\newtheorem{theorem}{Theorem}[section]

\newtheorem{lemma}[theorem]{Lemma}

\newtheorem{corollary}[theorem]{Corollary}

\theoremstyle{definition}
\newtheorem{definition}[theorem]{Definition}
\newtheorem{inductive step}[theorem]{Inductive step}

\newtheorem{inductive lemma}[theorem]{Inductive Lemma}
\theoremstyle{remark}
\newtheorem{problem}[theorem]{Problem}
\newtheorem{example}[theorem]{Example}
\newtheorem{remark}[theorem]{Remark}
\newtheorem*{remark*}{Remark}

\newtheorem{question}[theorem]{Question}

\newtheorem*{example*}{Example}

\renewcommand{\mod}{\mathrm{mod}} 

\newcommand{\wh}{\widehat}

\newcommand{\p}{\partial}

\newcommand{\Z}{{\mathbf{Z}}}
\newcommand{\R}{{\mathbf{R}}}
\newcommand{\C}{{\mathbf{C}}}

\newcommand{\codim}{{\rm codim}}

\newcommand{\cT}{\mathcal{T}}

\numberwithin{figure}{section}

\title{Caustics of Lagrangian homotopy spheres with stably trivial Gauss map}

\author{Daniel  \'Alvarez-Gavela}
\address{Department of Mathematics \\ Massachusetts Institute of Technology \\ Cambridge, MA, 02139}
\email{dgavela@mit.edu}
\thanks{DA was partially supported by the Simons Foundation}

\author{David Darrow}
\address{Department of Mathematics \\ Massachusetts Institute of Technology \\ Cambridge, MA, 02139}
\email{dddarrow@mit.edu}
\thanks{}

\begin{document}
\begin{abstract}
For each positive integer $n$, we give a geometric description of the stably trivial elements of the group $\pi_n U_n/O_n$. In particular, we show that all such elements admit representatives whose tangencies with respect to a fixed Lagrangian plane consist only of folds. By the h-principle for the simplification of caustics, this has the following consequence: if a Lagrangian distribution is stably trivial from the viewpoint of a Lagrangian homotopy sphere, then by an ambient Hamiltonian isotopy one may deform the Lagrangian homotopy sphere so that its tangencies with respect to the Lagrangian distribution are only of fold type. Thus the stable triviality of the Lagrangian distribution, which is a necessary condition for the simplification of caustics to be possible, is also sufficient. We give applications of this result to the arborealization program and to the study of nearby Lagrangian homotopy spheres.
	
\end{abstract}

\maketitle

 \onehalfspacing
 \tableofcontents
  \section{Introduction}\label{sec:intro}

\subsection{Main result}

Let $(M,\omega)$ be a $2n$-dimensional symplectic manifold, $\gamma \subset TM$ a Lagrangian distribution and $L \subset M$ a closed Lagrangian submanifold. In \cite{AG18b} the first author established the following h-principle: if $\gamma$ is homotopic to a Lagrangian distribution with respect to which $L$ only has fold type tangencies, then $L$ is Hamiltonian isotopic to a Lagrangian submanifold which only has fold type tangencies with respect to $\gamma$. 

This h-principle reduces the problem of eliminating higher tangencies to the underlying homotopical problem. 
In the present article we solve this homotopical problem in the case where $L$ has the homotopy type of a sphere. The central notion is that of stable triviality, which we now define. 

\begin{definition}\label{def: stab triv 1}
We say that $\gamma|_L$ is {\em stably trivial} if $\gamma|_L \oplus \R$ is homotopic to $TL \oplus \R$ as Lagrangian distributions in the symplectic vector bundle $\C^{n+1} \to TM|_L \oplus \C \to L$.
\end{definition}


Our main result is the following, where we assume $n > 1$ (for $n=1$ the problem is trivial).

\begin{theorem}\label{thm:intro}
	Let $L \subset M$ be a Lagrangian homotopy sphere in a symplectic manifold $(M,\omega)$ and $\gamma \subset TM$ a Lagrangian distribution. The tangencies of $L$ with respect to $\gamma$ can be simplified to consist of only folds via a Hamiltonian isotopy of $L$ if and only if $\gamma|_L$ is stably trivial.
\end{theorem}

\begin{remark} We observe:
	\begin{enumerate}
	\item Since $T(T^*L)|_L \simeq L \times \C^n$ as symplectic vector bundles, we may think of the homotopy class of $\gamma|_L$ as an element of $\pi_n \Lambda_n$, where $\Lambda_n$ the Grassmannian of Lagrangian planes in $\C^n$. From this viewpoint stable triviality is equivalent to asking that this element is in the kernel of the stabilization map $\pi_n \Lambda_n \to \pi_n \Lambda_{n+1}$.
			\item The hypothesis that $\gamma|_L$ is stably trivial is automatically satisfied if $n$ is congruent to $0,4,6$ or $7$ modulo $8$, since $\pi_n \Lambda_{n+1} = 0$ for those values of $n$.
			\item The subgroup $\ker(\pi_n \Lambda_n \to \pi_n \Lambda_{n+1})$ of $\pi_n \Lambda_n$ is always cyclic, in fact it is infinite cyclic for even $n$ and cyclic of order 2 for odd $n>1$. We will exhibit an explicit generator in each dimension, see see Remark \ref{rem: gen}.
	\end{enumerate}
\end{remark}

The homotopical problem underlying Theorem \ref{thm:intro}, which by the h-principle \cite{AG18b} is equivalent to Theorem \ref{thm:intro} itself, is to show that each element in the kernel of the stabilization map $\pi_n \Lambda_n \to \pi_n \Lambda_{n+1}$ admits a representative $S^n \to \Lambda_n$ which only has fold type tangencies with respect to some fixed but arbitrary Lagrangian plane $P \subset \C^n$. This is the problem that is addressed in the present article. We formulate this precisely as Theorem \ref{thm: main} below, after introducing the notion of a formal fold, which is a special case of Entov's notion of a chain of corank 1 Lagrangian singularities \cite{En97}.

The Lagrangian Grassmannian $\Lambda_n$ admits a description as the homogeneous quotient $U_n/O_n$ where $U_n$ is the unitary group and $O_n$ the orthogonal group. Thus the homotopy groups of $O_n$, $U_n$ and $\Lambda_n$ are related via the long exact sequence in homotopy associated to the Serre fibration $O_n \to U_n \to \Lambda_n$. These homotopy groups were computed by Bott in the stable range \cite{B59}. However, while $\pi_n U_n$ lies in the stable range, $\pi_{n-1}O_n$ does not, and neither does $\pi_n \Lambda_n$.

 In fact, $\pi_{n-1}O_n$ and $\pi_n \Lambda_n$ are the first nonstable homotopy groups of $O_n$ and $\Lambda_n$, i.e. as soon as we stabilize them once we enter the stable range. Moreover, the stabilization maps $\pi_{n-1}O_n \to \pi_{n-1}O_{n+1}$ and $\pi_n \Lambda_n \to \pi_n \Lambda_{n+1}$ are epimorphisms. These groups lie in the so-called metastable range, which is somewhat more subtle than the stable range, but has also been studied in the literature and exhibits a secondary form of 8-fold periodicity for $n \geq 8$. In particular $\pi_n \Lambda_n$ has been computed \cite{K78}, and this computation is essential input for our approach.

When $n$ is even the problem is simpler because $\pi_n U_n=0$, as was already observed in \cite{AG18b}. The main novelty of the present article is to tackle the case of $n$ odd. The special cases $n=3,7$ are particularly subtle due to the parallelizability of $S^n$ and need to be addressed individually. We tackle the special cases $n=3,7$ by making explicit use of the geometry of the quaternions and octonions respectively. The key homotopical input is the well-known fact that multiplication by unit quaternions (resp. octonions), thought of as an element of $\pi_3O_4$ (resp. $\pi_7O_8)$, maps under stabilization to a generator of $\pi_3O$ (resp. $\pi_7O$). In the case $n=3$ we also sketch an alternative argument using Entov's technique of surgery of corank 1 Lagrangian singularities. 

We will give some applications of Theorem \ref{thm:intro} in Section \ref{sec:intro-app} below, but first we will briefly discuss what kind of results one might hope for in the general case where the distribution $\gamma|_L$ is not assumed to be stably trivial.

\subsection{Homotopically essential caustics of Lagrangian spheres}  

In order to go beyond the results of the present article and achieve a full classification of the homotopically essential caustics of Lagrangian spheres with respect to an arbitrary Lagrangian distribution it will be necessary to understand the geometry of the elements of $\pi_n \Lambda_n$ coming from the generators of the stable groups $\pi_n U_n \simeq \pi_n U$, since these elements are in general not stably trivial. The group $\pi_n U$ is of course well understood from Bott periodicity: it is isomorphic to $\Z$ for $n$ odd and it is trivial for $n$ even. 

While there exist explicit descriptions of the generators of the groups $\pi_{2k-1} U$, for example see \cite{PR03} for simple formulas in $\pi_{2k-1} U_{2^{k-1}}$, these formulas become quite complicated after de-stabilizing down to $\pi_{2k-1} U_{2k-1}$. In particular it is not clear what type of singularities of tangency one obtains, or to what extent they can be simplified. 

\begin{problem}
For each odd integer $n$, exhibit an explicit representative for a generator of $\pi_n U_n$ so that the corresponding element of $\pi_n \Lambda_n$ has the simplest possible tangencies with respect to a fixed but arbitrary Lagrangian plane $P \subset \C^n$.
\end{problem}

For example, when $n=3$ a generator of $\pi_3 \Lambda_3 \simeq \Z/4$, which is the image of a generator of $\pi_3 U_3$ since $\pi_2O_3=0$, admits a representative which has folds along a torus $T \subset S^3$ and pleats along a $(1,1)$ curve on $T$, where we embed the torus in $S^3$ as the boundary of a standard handlebody. See Figure \ref{11curve}, as well as Remark \ref{rem: alt}.

\begin{figure}[h]
\includegraphics[scale=0.45]{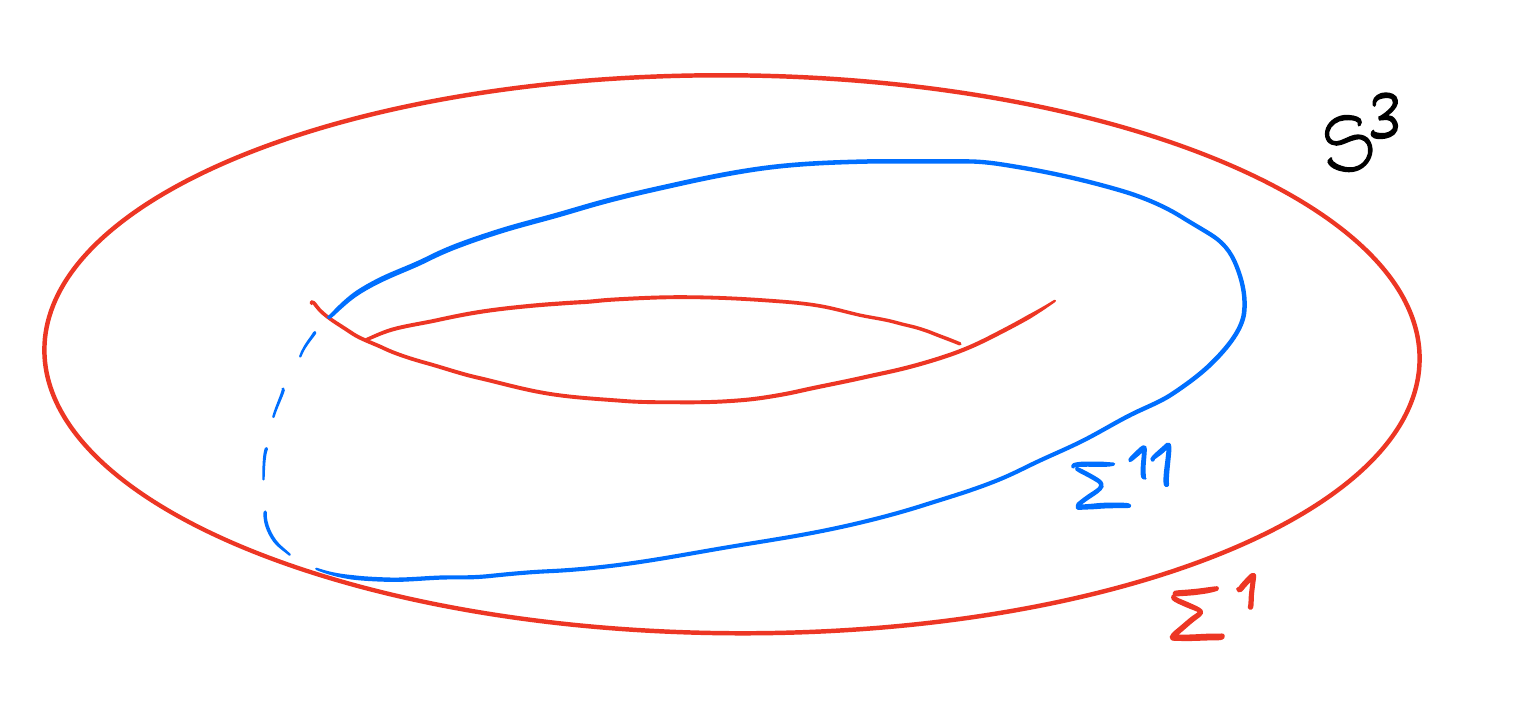}
\caption{The chain of singularities for a generator of $\pi_3 \Lambda_3$.}
\label{11curve}
\end{figure}

For $n=5$ it is not known to us how simple of a tangency locus one can achieve for the image of a generator of $\pi_5 U_5$ in $\pi_5 \Lambda_5 \simeq \Z \oplus \Z/2$.

Moreover, note that understanding the image of $\pi_n U_n \to \pi_n \Lambda_n$ would not be by itself sufficient to achieve a full classification of the homotopically essential caustics of Lagrangian spheres. As the simplest example consider the case $n=2$, where we have $\pi_2 U_2=0$ but the subgroup of stably trivial elements in $\pi_2 \Lambda_2 \simeq \Z$ has index 2. In this case the situation is not so bad: a generator of $\pi_2 \Lambda_2$ admits a representative with a circle of folds and a single pleat at a point on the circle, see Figure \ref{notfold}. However in general it is not clear to us what one should expect.

\begin{figure}[h]
\includegraphics[scale=0.6]{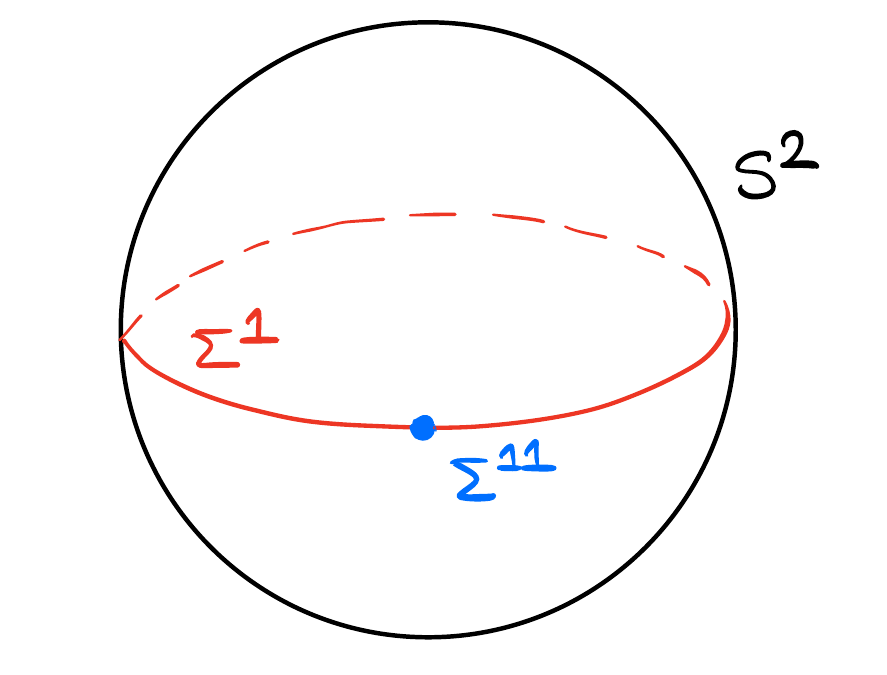}
\caption{The chain of singularities for a generator of $\pi_2 \Lambda_2$.}
\label{notfold}
\end{figure}

The most optimistic hope is that it is always possible to find $\Sigma^2$-nonsingular representatives. While this seems unlikely, we do not know of a counterexample. Hence we formulate the following:

\begin{question}
	Do all elements of $\pi_n \Lambda_n$ admit representatives whose intersection with a fixed Lagrangian plane $P \subset \C^n$ has dimension $\leq 1$?
\end{question}

If we set $\Sigma^{2} = \{ Q \in \Lambda_n : \, \dim(P \cap Q) \geq 2 \} \subset \Lambda_n$ for a fixed but arbitrary Lagrangian plane $P \subset \C^n$ whose choice is immaterial, then the above question is equivalent to asking whether the inclusion $(\Lambda_n \setminus \Sigma^{2}) \subset \Lambda_n$ induces a surjection $ \pi_n(\Lambda_n \setminus \Sigma^{2} )\to \pi_n \Lambda_n$. We note that the inclusion $(\Lambda_n \setminus \Sigma^{2} ) \subset \Lambda_n$ is far from being a homotopy equivalence, as can be easily deduced from the cohomological calculations in the literature. 

For example if $n=3$, denote by $D^3 \subset \R^3$ the unit disk and let $D^3 \to \Lambda_3$ be the Gauss map of a neighborhood of a generic isolated Lagrangian $\Sigma^2$ singularity \cite{AGV85}. Then the resulting element in $\pi_3(\Lambda_3, \Lambda_3 \setminus \Sigma^{2})$ can be shown to be non-trivial by means of a characteristic class in $H^3(\Lambda_3;\Z/2)$ which is Poincar\'e dual to the codimension 3 cycle $\Sigma^{2} \subset \Lambda_3$, see \cite{A67}. 

In fact for any positive integer $n$ the integral cohomology ring of $\Lambda_n$ is generated by characteristic classes dual to similarly defined cycles \cite{F68}. While these classes may be used to prove results establishing the necessity of higher singularities, a different method will most likely be needed to prove results in the opposite direction.

\subsection{Applications}\label{sec:intro-app}

We present two applications of our main result Theorem \ref{thm:intro}, one to the arborealization program and another to the study of nearby Lagrangian homotopy spheres.

\subsubsection{Arborealization program}

As our first application we give a simple proof that polarized Weinstein manifolds which are obtained from the standard Darboux ball by a single handle attachment admit arboreal skeleta. This recovers a special case of the main theorem of \cite{AGEN20b}, where it is shown more generally that any polarized Weinstein manifold admits an arboreal skeleton. The argument used in \cite{AGEN20b} is rather involved due to the subtleties arising from the interaction of three or more strata, whereas for the special class of polarized Weinstein manifolds obtained from a single handle attachment one can give a rather simple argument. Namely, the proof consists of a direct application of Theorem \ref{thm:intro} together with Starkston's local model for the arborealization of a semi-cubical cusp \cite{St18}, which was used in that paper to arborealize Weinstein manifolds of dimension four.


In addition to the simplicity of the argument, a novel feature of the result we establish  is that the arboreal skeleton we end up with has arboreal singularities of a particularly simple type. This conclusion does not follow directly from \cite{AGEN20b}.

Before we state the result, recall that arboreal singularities are modeled on rooted trees equipped with a decoration of a sign $\pm1$ for each edge not adjacent to a root \cite{St18, AGEN20a}. By the height of a vertex we mean the number of edges between that vertex and the root, by the height of a tree we mean the maximal height among all vertices and by the height of an arboreal singularity we mean the height of the corresponding signed rooted tree.

\begin{corollary}
	Let $(W, \lambda, \phi)$ be a Weinstein manifold such that $TW$ admits a global field of Lagrangian planes and such that the Morse Lyapunov function $\phi$ only has two critical points. Then by a homotopy of the Weinstein structure we can arrange it so that the skeleton of $(W, \lambda)$ becomes arboreal, and moreover so that the arboreal singularities which appear in the skeleton have height $ \leq 2$.
\end{corollary}

We briefly describe the proof, which follows the blueprint of \cite{St18}. First one blows up a Darboux ball around the origin into the cotangent bundle of a Morse-Bott disk. The stable manifold of the other critical point then lands on this Morse-Bott disk along a front projection. The singularities of this front are a priori very complicated, but existence of a polarization is precisely the homotopical input needed for Theorem \ref{thm:intro} to apply. Hence by a Legendrian isotopy of the attaching Legendrian, which can be realized by a homotopy of the Weinstein structure, we may assume that the front only has semi-cubical cusp singularities. Finally the cusps can be traded for arboreal singularities as shown in \cite{St18}.

\subsubsection{Nearby Lagrangian homotopy spheres}   

As our second application we show that any nearby Lagrangian in the cotangent bundle of a homotopy sphere can be deformed via a Hamiltonian isotopy so that it is generated by a framed generalized Morse family on some bundle of tubes. We briefly explain the terminology before formally stating the result.

Following Igusa \cite{I87}, a framed generalized Morse family, or {\em framed function} for short, on the total space of a fibre bundle $W\to M$ is a function $f:W \to \R$ such that the restriction of $f$ to each fibre is Morse or generalized Morse (i.e. we allow cubic birth/death of Morse critical points), and moreover such that the negative eigenspaces of the fibrewise Hessian at the fibrewise critical points are equipped with framings which vary continuously over $M$ and are suitably compatible at the birth death points. 

Following Waldhausen \cite{W82}, tubes are codimension zero submanifolds with boundary $T \subset \R^{n+1}$ which up to a compactly supported isotopy are given by the standard model for a smooth handle attachment on the boundary of the half space $\{x_{n+1} \leq 0\}$. A {\em tube bundle} is a fibre bundle of tubes $T \to W \to M$ where we assume that all tubes are contained in a fixed Euclidean space, i.e. $W \subset M \times \R^{n+1}$ and $W \to M$ is the restriction of the obvious projection $M \times \R^{n+1} \to M$.

We can now state:

\begin{corollary}\label{cor: ap1}
	Let $\Sigma_0,\Sigma_1$ be homotopy spheres and $\Sigma_1 \subset T^*\Sigma_0$ a Lagrangian embedding. There exists a Hamiltonian isotopy $\varphi_t$ of $\Sigma_1$ such that $\varphi_1(\Sigma_1)$ is generated by a framed function on some tube bundle $T \to W \to \Sigma_0$.
\end{corollary}

The starting point of the argument is the recent article \cite{ACGK20} of Abouzaid, Courte, Guillermou and Kragh, where it is shown that if $\Sigma_0, \Sigma_1$ are homotopy spheres and $\Sigma_1 \subset T^*\Sigma_0$ is a Lagrangian embedding, then there exists a tube bundle $W \to \Sigma_0$ such that $\Sigma_1$ is generated by a function $f:W \to \R$. 

In particular it follows from their result that the stable Gauss map $\Sigma_1 \to U/O$ is trivial, which unwinding the definition means that Theorem \ref{thm:intro} applies to $M=T^*\Sigma_0$, $L=\Sigma_1$ and $\gamma= \ker(d \pi)$ the vertical distribution, where $\pi: T^*\Sigma_0 \to \Sigma_0$ is the cotangent bundle projection. Therefore, $\Sigma_1$ can be deformed by a Hamiltonian isotopy $\varphi_t$ so that $\varphi_1(\Sigma_1)$ only has fold tangencies with respect to the vertical distribution. 

After replacing the tube bundle $W$ with an appropriate stabilization, by the homotopy lifting property for generating families \cite{Si86} it is possible to cover the isotopy $\varphi_t(\Sigma_1)$ by a homotopy of generating functions. At the end of the isotopy we obtain a generating function for $\varphi_1(\Sigma_1)$ which only has Morse or Morse birth/death critical points. Indeed, Morse critical points correspond to points where $\varphi_1(\Sigma_1)$ and the vertical distribution are transverse and Morse birth/death critical points correspond to fold type tangencies. 

Finally, this function may not admit a framing but one can fix this by further replacing $W$ with a twisted stabilization of $W$ using the fact that the projection $\varphi_1(\Sigma_1) \to \Sigma_0$ is a homotopy equivalence \cite{A12}.

Framed functions are the homotopically canonical way of studying fibre bundles via parametrized Morse theory \cite{I02}. It is our hope that Corollary \ref{cor: ap1} may be useful for the study of nearby Lagrangians via parametrized Morse theory.

\subsection{Structure of the article}
In Section \ref{sec: chain} we introduce the notion of a formal fold and translate the geometric problem into a homotopical problem. In Section \ref{sec: compute} we perform the homotopical calculation necessary to establish our main theorem in dimensions not equal to 3 or 7. In Section \ref{sec: excep} we tackle the special dimensions 3 and 7. In Section \ref{sec: app} we give the proofs of the applications stated above.

\subsection{Acknowledgements} This article was the outcome of an undergraduate Research Opportunities Program (UROP) which the second author undertook at MIT under the supervision of the first author. We are grateful to the UROP program for enabling this collaboration. We would also like to thank Sasha Givental for useful conversations.

\section{Formal folds}\label{sec: chain}

\subsection{Tangencies of fold type}

\subsubsection{Lagrangian tangencies}

Let $(M, \omega)$ be a $2n$-dimensional symplectic manifold, $L \subset M$ a smooth Lagrangian submanifold and $\gamma \subset TM$ a Lagrangian distribution.

\begin{definition}
	A {\em tangency} between $L$ and $\gamma$ is a point $x \in L$ such that $T_x L \cap \gamma_x \neq 0$.
\end{definition}

If $\gamma=\ker(d \pi)$ for a Lagrangian fibration $\pi : M \to B$, then tangencies of $L$ with respect to $\gamma$ are the same as singular points of the restriction $\pi|_{L} : L\to B$, i.e. points $x \in L$ at which the differential $d \pi_x : T_x L \to T_{\pi(x)} B$ fails to be an isomorphism. If $L$ is exact then we may lift it to a Legendrian $\wh L$ in the contactization $M \times \R$ and the tangencies of $L$ with respect to $\gamma$ can also be thought of as the singularities of the front $\wh L \to B \times \R$, which is known as the caustic in the literature \cite{A90}. 

A  tangency point $x \in L$ is said to be of corank 1, or $\Sigma^2$-nonsingular, if $\dim( T_x L \cap \gamma_x) = 1$. The locus of corank 1 tangencies $\Sigma^1 = \{ x \in L : \,  \dim(L \cap \gamma) = 1 \} $ is $C^\infty$-generically a smooth hypersurface in $L$ and $\ell = (TL \cap \gamma)|_{\Sigma^1}$ is a line field inside $TL|_{\Sigma^1}$. We say that $\gamma$ is $\Sigma^2$-nonsingular if all its tangencies with $L$ are $\Sigma^2$-nonsingular, so the tangency locus of $L$ with $\gamma$ is equal to $\Sigma^1$, which in this case is $C^\infty$-generically a smooth, closed hypersurface in $L$ without boundary. 

 While $C^\infty$-generic Lagrangian tangencies are non-classifiable, the class of $\Sigma^2$-nonsingular tangencies does admit a finite list of local models, at least in the case where $\gamma$ is integrable \cite{AGV85}. The simplest type of $\Sigma^2$-nonsingular tangency is called a fold. This is the only type of tangency we will need to consider in the present article. 
 
\begin{definition} We say that a tangency point $x \in \Sigma^1$ is of {\em fold type} if $\Sigma^1$ is transversely cut out in a neighborhood of $x$ and $\ell_x \pitchfork T_x\Sigma^1$ inside $T_x L$.
\end{definition}

When $\gamma$ is integrable, a fold tangency is locally symplectomorphic to the normal form
\[ L = \{ q= p^2 \} \times \R^{n-1} \subset T^*\R \times T^*\R^{n-1} , \qquad  \gamma = \ker( T^*\R^n \to \R^n). \]

\begin{figure}[h]
\includegraphics[scale=0.7]{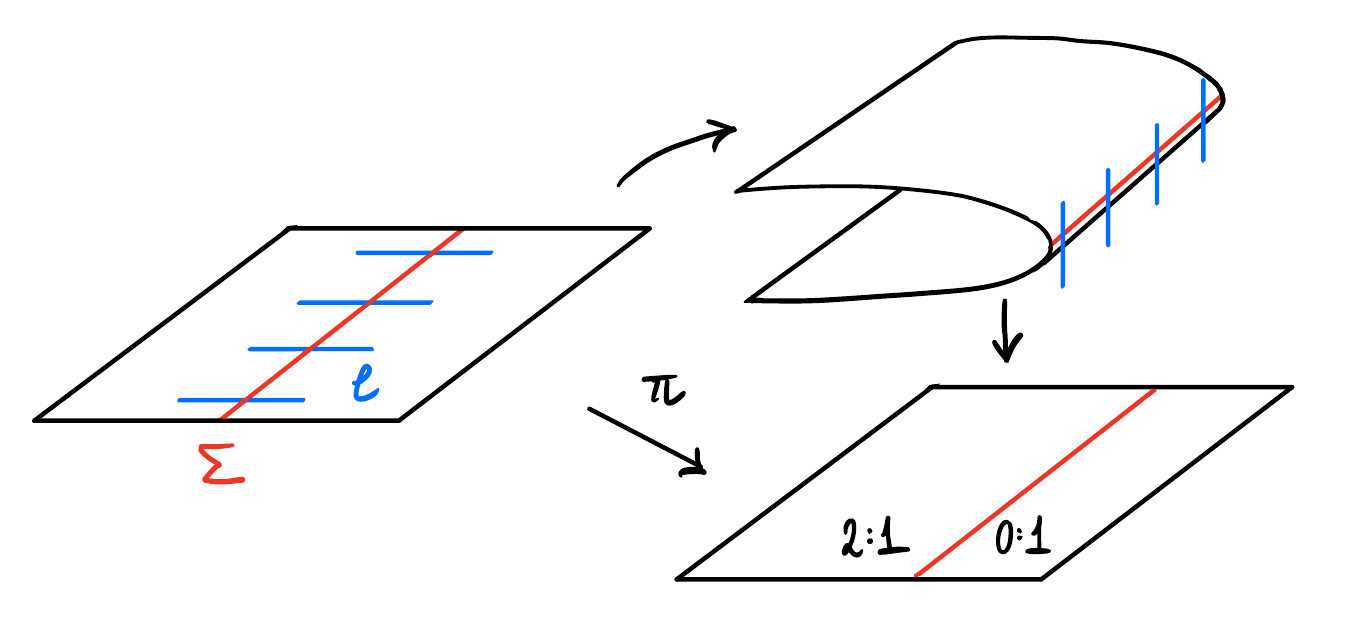}
\caption{A fold type tangency. Note that $\ell= \ker(d \pi|_L)$ is transverse to the singular locus $\Sigma \subset L$.}
\label{foldnatomy}
\end{figure}

\begin{remark} We note that in the contactization, fold tangencies correspond to semi-cubical cusps of the Legendrian front. \end{remark}

\subsubsection{The h-principle for the simplification of caustics}

In order to reduce Theorem \ref{thm:intro} to a homotopical problem, we use the h-principle for the simplification of caustics established by the first author in \cite{AG18b}. It states the following:

\begin{theorem}[\cite{AG18b}]\label{thm:h-p} Let $(M,\omega)$ be a symplectic manifold, $L \subset M$ a Lagrangian submanifold and $\gamma \subset TM$ a Lagrangian distribution. Suppose that $\gamma$ is homotopic through Lagrangian distributions to a Lagrangian distribution with respect to which $L$ only has fold tangencies. Then $L$ is Hamiltonian isotopic to a Lagrangian submanifold which only has fold tangencies with respect to $\gamma$.
\end{theorem}

Hence to prove Theorem \ref{thm:intro} it suffices to show that under the stated hypotheses $\gamma$ is homotopic to a Lagrangian distribution which only has fold tangencies with $L$. 

\begin{remark}
The hypothesis in Theorem \ref{thm:h-p} only cares about the restriction of $\gamma$ to $L$, since any homotopy of $\gamma|_L$ can be extended to a homotopy of $\gamma$. Furthermore, by taking a Weinstein neighborhood of $L$ we may immediately reduce to the case $M=T^*L$, which is therefore the only case we will consider in what follows.
\end{remark}

\subsection{Formal folds and their stable triviality} 
\subsubsection{Formal folds}

The homotopical object underlying a Lagrangian distribution with only fold type tangencies is a formal fold, which is defined as follows:

\begin{definition}
A \emph{formal fold} in a smooth manifold $L$ consists of a pair $(\Sigma, v)$, where $\Sigma \subset L$ is a co-orientable smooth closed hypersurface in $L$ and $v$ is a choice of co-orientation of $\Sigma$. \end{definition}

\begin{figure}[h]
\includegraphics[scale=0.6]{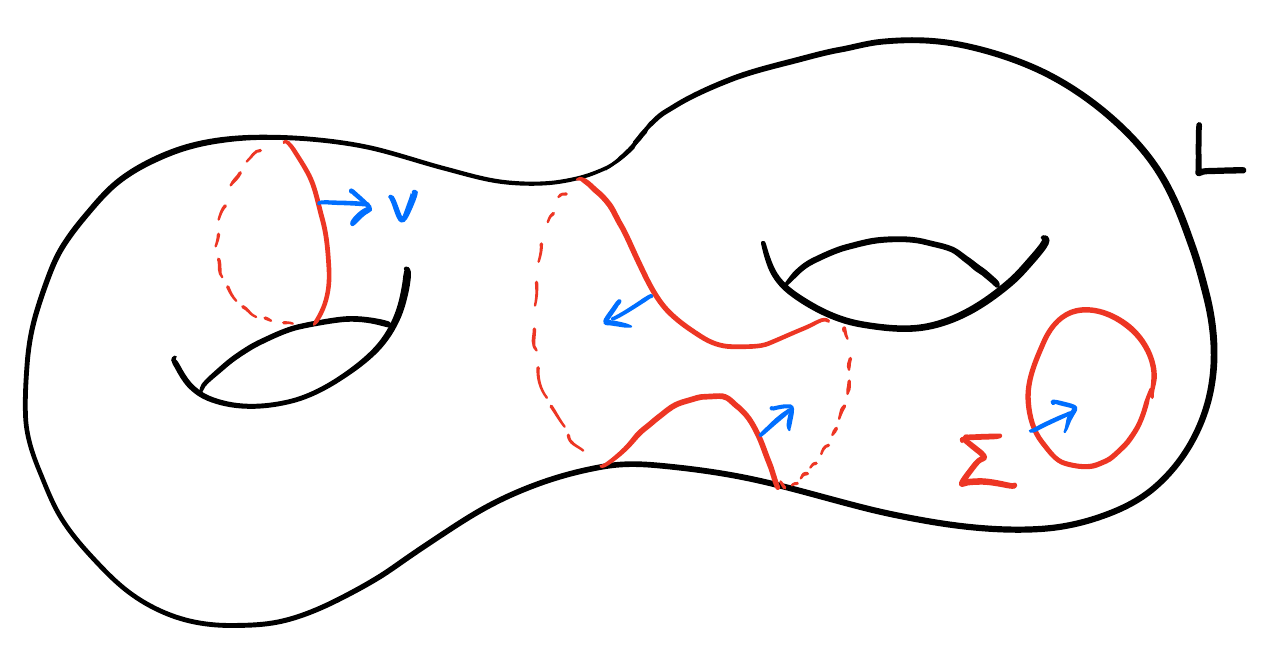}
\caption{A formal fold on a surface consists of a disjoint union of co-oriented simple closed curves.}
\label{formalold}
\end{figure}

\begin{remark} Formal folds are the simplest version of the notion of a chain of Lagrangian singularities as defined by Entov \cite{En97}, generalizing the notion of a chain of singularities for smooth maps \cite{E72}. We will not need this more general notion in what follows and hence will not discuss it further, with the exception of the non-essential Remark \ref{rem: alt}. \end{remark}

Let $\gamma\subset T(T^*L)|_{L}$ be a Lagrangian distribution which has only fold type tangencies with respect to $L$. That is, the intersection $\gamma_x \cap T_xL \subset T_x(T^*L)$ has dimension $\leq 1$ for any $x\in L$, the subset $\Sigma = \{ x\in L :\,  \gamma_x \cap T_xL \neq 0 \} \subset L$ is a transversely cut out hypersurface and $\ell = \gamma|_\Sigma \cap TL$ is a line field along $\Sigma$ which is transverse to $\Sigma$. To such a $\gamma$ we associate a formal fold $(\Sigma, v)$ by specifying $v$ to be the Maslov co-orientation \cite{A67, En97}.

Conversely, if $(\Sigma, v)$ is a formal fold on $L$, there is a homotopically unique Lagrangian distribution $\gamma(\Sigma,v) \subset T(T^*L)|_{L}$ which has only fold type tangencies with respect to $L$ and whose associated formal fold is $(\Sigma, v)$. For existence, let $\Sigma \times [-1,1]$ be a tubular neighborhood of $\Sigma$ in $L$ such that the coordinate $t \in [-1,1]$ is compatible with the co-orientation, i.e. $v= \p / \p t$. On $L \setminus ( \Sigma \times [-1,1])$ we define $\gamma(\Sigma, v)$ to be the vertical distribution. On $\Sigma \times [-1,1]$ we define it to be the direct sum of the vertical distribution in $T(T^*\Sigma)|_{\Sigma}$ and the line field $\ell \subset T(T^*[-1,1])|_{[-1,1]}$ defined by
\[ \ell_t = \text{span}\big\langle\sin\left( \frac{\pi t}{ 2}\right) \frac{\p }{\p u} + \cos\left(\frac{ \pi t }{2}\right)\frac{\p}{\p t} \big\rangle \subset T_t(T^*[-1,1]) \]
where $u$ is the momentum coordinate dual to $t$, see Figure \ref{movingline}.

\begin{figure}[h]
\includegraphics[scale=0.6]{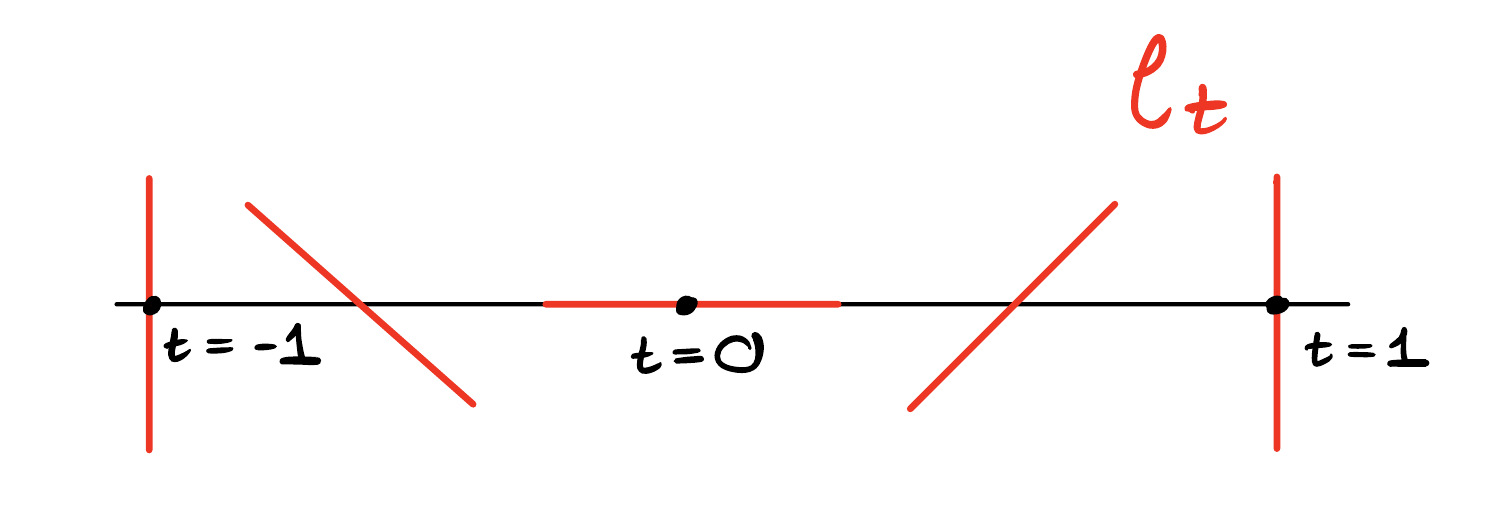}
\caption{The line field $\ell_t$.}
\label{movingline}
\end{figure}

The fact that $\gamma(\Sigma,v)$ is homotopically unique is straightforward to verify using the well-known fact that that the space of Lagrangian planes in $\C^n$ which are transverse to a fixed Lagrangian plane $P$ is contractible; indeed this space can be identified with the (convex) space of quadratic forms on any Lagrangian plane $Q$ which is transverse to $P$. 

Finally, we note that the homotopy class of $\gamma$ only depends on the formal fold $(\Sigma, v)$ up to ambient isotopy in $L$.

\subsubsection{Stable triviality of formal folds}


Let $\gamma \subset T(T^*L)|_L$ be a Lagrangian distribution defined along $L$. We say that $\gamma$ is trivial if it is homotopic through Lagrangian distributions to the vertical distribution, which is defined to be $\nu = \ker(d \pi)$ for $\pi:T^*L \to L$ the cotangent bundle projection. More generally:

\begin{definition}\label{def: stab triv 2}
We say that $\gamma$ is {\em stably trivial} if $\gamma \oplus \R$ and  $\nu \oplus \R$ are homotopic as Lagrangian distributions in $T(T^*L)|_L \oplus \C$. 
\end{definition}

\begin{remark}
This notion of stable triviality is equivalent to the one given in Definition \ref{def: stab triv 1} since $TL$ and $\nu$ are homotopic Lagrangian distributions in $T(T^*L)|_L$. For example, this can be seen by rotating one to the other via a compatible almost complex structure $J$ on $T^*L$ such that $J T_xL = \nu_x$ in $T_x(T^*L)$ for all $x \in L$.
\end{remark}

\begin{lemma}\label{lem: folds are stably trivial}
Let $(\Sigma, v)$ be a formal fold in $L$. Then $\gamma(\Sigma,v)$ is stably trivial.
\end{lemma}

\begin{proof}
Consider the path $p_1:[-1,1] \to U_2$ given by
\[. t \mapsto
 \begin{pmatrix}
e^{ \pi i (t+1) /2} &0 \\
0 & 1 
\end{pmatrix} \]
and the path $p_2:[-,1,1] \to U_2$  given by
\[. t \mapsto
 \begin{pmatrix}
1 &0 \\
0 & e^{ \pi i (t+1) /2}
\end{pmatrix} \]

Post-composing $p_1$ and $p_2$ with the projection $U_2 \to \Lambda_2$ (i.e. taking the images $p_i(t)(\R^2) \subset \C^2$) we obtain loops $\eta_1,\eta_2: ([-1,1],\{ \pm 1 \}) \to (\Lambda_2, \R^2)$, i.e. $\eta_i(-1)=\eta_i(1)=\R^2$ for $i=1,2$. Since the isomorphism $\pi_1 \Lambda_2 \simeq \Z$ is induced by $\text{det}^2:\Lambda_2 \to U_1$ and $\det^2 \circ \eta_1 = \det^2 \circ \eta_2$ (both are equal to the function $e^{\pi i (t+1)})$ it follows that $\eta_1$ and $\eta_2$ are homotopic relative to $\{ \pm 1\}$, as can be verified explicitly.

At a point $x \in \Sigma \times [-1,1]$ we may split $T_xL \oplus \C  \simeq T_x\Sigma \oplus \C^2$. From the above observation it follows that $\gamma(\Sigma, v) \oplus \R$ is homotopic to the distribution $\nu \oplus \wh \ell$, where $\wh \ell$ denotes the line field in $S^n \times \C$ defined as $i \R$ outside of $\Sigma \times [-1,1]$ and for $(x,t) \in \Sigma \times [-1,1]$ given by
\[\wh \ell = \text{span}\big\langle\sin\left( \frac{\pi t}{ 2}\right) \frac{\p }{\p x} + \cos\left(\frac{ \pi t }{2}\right)\frac{\p}{\p y} \big\rangle  \subset \C. \]

But every map $S^n \to \Lambda_1$ is null-homotopic when $n>1$ since $\Lambda_1 \simeq S^1$. Hence $\wh \ell$ is homotopic to the trivial distribution $S^n \times \R \subset S^n \times \C$ and consequently $\gamma(\Sigma,v) \oplus \R$ is homotopic to $\nu \oplus \R$, which was to be proved. \end{proof}

\subsection{Reduction to homotopy theory}

\subsubsection{Formal folds in $\R^n$}

Let $(\Sigma,v)$ be a formal fold in $\R^n$. We assume $\Sigma \subset \R^n$ to be compact, hence the corresponding Lagrangian distribution $\gamma(\Sigma,v) \subset T(T^*\R^n)|_{\R^n}$ is vertical at infinity. In other words, $\gamma(\Sigma,v)$ is equal to the vertical distribution $\nu=\ker(d \pi)$ outside of a compact subset, where $\pi:T^*\R^n \to \R^n$ is the standard projection. 

Since $T(T^*\R^n)|_{\R^n} \simeq \R^n \times \C^n$ as symplectic vector bundles, there is a one to one correspondence between homotopy classes of Lagrangian distributions in $T(T^*\R^n)|_{\R^n}$ which are vertical at infinity and elements of $\pi_n \Lambda_n$, where $\Lambda_n $ is the Grassmannian of linear Lagrangian subspaces of $\C^n$. Thus to a formal fold $(\Sigma,v)$ in $\R^n$ is associated an element $\alpha(\Sigma,v) \in \pi_n \Lambda_n$. Here we think of the $n$-sphere as the one-point compactification of $\R^n$ with the basepoint at infinity and we take the (vertical) imaginary plane $i \R^n$ as the basepoint of $\Lambda_n$.

By Lemma \ref{lem: folds are stably trivial}, every element of the form $\alpha(\Sigma,v)$ is in the kernel of the stabilization map $\pi_n \Lambda_n \to \pi_n \Lambda_{n+1}$ induced by the inclusion $\Lambda_n \hookrightarrow \Lambda_{n+1}$, which we recall is given by $$P \mapsto P \oplus \R \subset \C^n \oplus \C, \qquad P \subset \C^n.$$

In Sections \ref{sec: compute} and \ref{sec: excep} below we will prove that the converse is also true:

\begin{theorem}\label{thm: main} Every element of $\ker( \pi_n \Lambda_n \to \pi_n \Lambda_{n+1})$ admits a representative of the form $\alpha(\Sigma,v)$ for some formal fold $(\Sigma,v)$ in $\R^n$.\end{theorem}

In the rest of the present section we will show how Theorem \ref{thm: main}, together with the h-principle \ref{thm:h-p}, imply our main result Theorem \ref{thm:intro}.

\subsubsection{Formal folds in homotopy spheres}

Let $L$ be an $n$-dimensional homotopy sphere and denote by $\C^n \to E_L \to L$ the symplectic vector bundle $T(T^*L)|_L$. Let $\Lambda_n \to \Lambda_L \to L$ denote the associated Grassmann bundle, whose fibre over $x \in L$ is the Grassmannian of linear Lagrangian subspaces of $T_x(T^*L)$. Let $f:D^n \to L$ be a smooth embedding of the closed unit disk $D^n = \{ \|x\| \leq 1 , \, \, x \in \R^n \}$, which is unique up to isotopy. 

After identifying the interior $B^n=\{ \| x \|<1, \, \, x \in \R^n\}$ of $D^n$ with $\R^n$, the embedding $f$ induces a map
\[ f_*:\pi_n \Lambda_n \to \pi_0 \Gamma (\Lambda_L), \]
where $\Gamma(\Lambda_L)$ denotes the space of sections of $\Lambda_L$. This is induced from a pushforward at the level of spaces, i.e. from the space of maps $(D^n, \p D^n) \to (\Lambda_n, i\R^n)$ to the section space $\Gamma(\Lambda_L)$. Explicitly, a Lagrangian distribution in $T(T^*D^n)|_{D^n}$ which is vertical near $\p D^n$ is extended to $T(T^*L)|_L$ as the vertical distribution outside of $f(D^n)$. Note that at the level of spaces the pushforward takes formal folds to formal folds, see Figure \ref{implant}. 

\begin{figure}[h]
\includegraphics[scale=0.65]{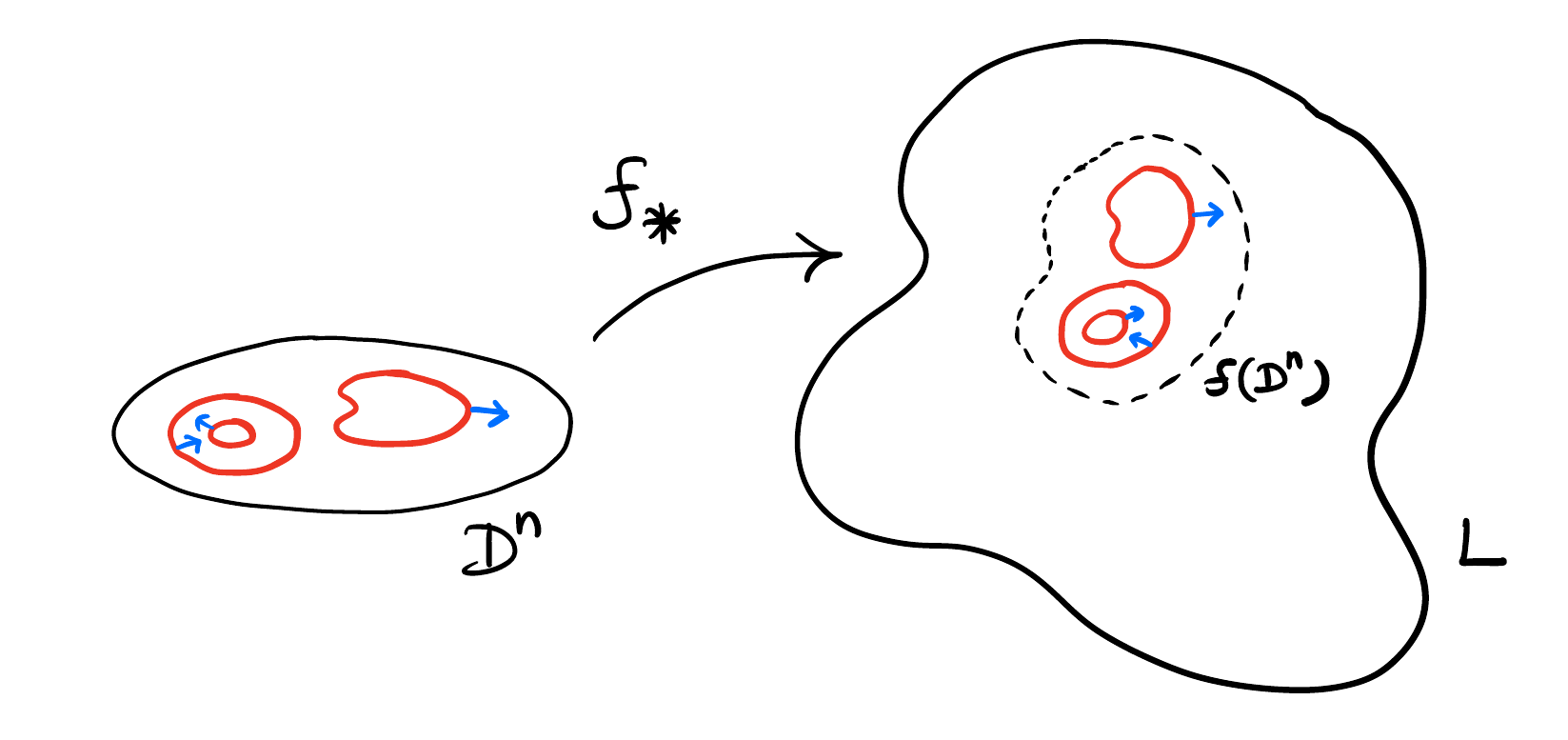}
\caption{Pushing forward a formal fold from $D^n$ to $L$ via $f$.}
\label{implant}
\end{figure}

\begin{lemma} $f_* \pi_n \Lambda_n = \pi_0 \Gamma(\Lambda_L)$.
\end{lemma}

\begin{proof}
Any  Lagrangian distribution $\gamma \subset T(T^*L)|_L$  may be deformed so that it is equal to the vertical distribution on a neighborhood $U$ of $L \setminus f(B^n)$ since $L \setminus f(B^n)$ is contractible. 
\end{proof}

Denote by $\Gamma^{\text{st}}(\Lambda_L) \subset \Gamma(\Lambda_L)$ the subspace of stably trivial sections. It is clear that $f_* \ker( \pi_n \Lambda_n \to \pi_n \Lambda_{n+1}) \subset \pi_0 \Gamma^{\text{st}}(\Lambda_L)$. Again we have surjectivity:

\begin{lemma}\label{lem:surj of model} $f_* \ker( \pi_n \Lambda_n \to \pi_n \Lambda_{n+1}) =\pi_0 \Gamma^{\text{st}}(\Lambda_L)$. \end {lemma}

\begin{proof} If a Lagrangian distribution $\gamma \subset T(T^*L)|_L$ which is vertical in a neighborhood $U$ of $L \setminus f(B)$ is stably trivial, then $\gamma \oplus \R$ and $\nu \oplus \R$ are homotopic in $T(T^*L)|_L \oplus \C$, but the homotopy need not be fixed in $U$. So we need to fix this.

We may assume that $U$ itself is contractible, for example we can set $U=L \setminus f( \frac{1}{2}B^n)$ for $\frac{1}{2}B^n=\{ \|x\|<1/2, \,\ \, x \in \R^n\}$. Let $x_0 \in U$ be a point outside of $f(D^n)$. The restriction of the homotopy between $\gamma \oplus \R$ and $\nu \oplus \R$ to $x_0$ determines an element of $\pi_1 \Lambda_{n+1}$. Now, $\pi_1 \Lambda_n \to \pi_1 \Lambda_{n+1}$ is an isomorphism for any $n \geq 1$, hence after a suitable deformation of $\gamma$ we may assume that this homotopy is through Lagrangian planes of the form $\gamma_t(x_0)  \oplus \R$, where $\gamma_t(x_0) \subset T_{x_0}(T^*L)|_L$. 

We may then use the homotopy $\gamma_t(x_0)$ to further deform $\gamma$ so that it is equal to the vertical distribution $\nu$ at the point $x_0$ and so that $\gamma \oplus \R$ is homotopic to $\nu \oplus \R$ through distributions which are equal to $\nu \oplus \R$ at the point $x_0$. Explicitly, trivialize a neighborhood $V \simeq \R^n$ of $x_0$ contained in $U$, first deform $\gamma(x)$ so that it is constant and equal to $\gamma(x_0)$ in that neighborhood, then replace it with $\gamma_{\phi(x)} (x_0)\subset \C^n$ where $\phi:\R^n \to [0,1]$ is a compactly supported function such that $\phi(x_0)=1$.

Finally, since $U$ is contractible we may further deform $\gamma$ so that the same property holds over all of $U$, i.e. $ \gamma$ is vertical over $U$ and $ \gamma \oplus \R$ is homotopic to $\nu \oplus \R$ through distributions which are equal to $\nu \oplus \R$ over $U$. This proves the lemma.
\end{proof}

We are now ready to prove our main result.

\begin{proof}[Proof of Theorem \ref{thm:intro} assuming Theorem \ref{thm: main}] The necessity of stable triviality follows from Lemma \ref{lem: folds are stably trivial}. The sufficiency of stable triviality follows from Lemma \ref{lem:surj of model} and Theorem \ref{thm: main}, which show that the hypothesis needed to apply the h-principle Theorem \ref{thm:h-p} is satisfied. 
\end{proof}

It therefore remains for us to establish Theorem \ref{thm: main}. This will be achieved in Section \ref{sec: compute} for those dimensions not equal to 3 or 7 and in Section \ref{sec: excep} for the exceptional dimensions 3 and 7.

  \section{Homotopical computation}\label{sec: compute}

\subsection{Homotopical background}\label{sec:tables}

We begin by reviewing some relevant background in homotopy theory, in particular we review for future reference certain stable and nonstable homotopy groups of the unitary and orthogonal groups and of their homogeneous quotient, the Lagrangian Grassmannian. 

\subsubsection{The classical groups}\label{sec: review}
Recall that to a Serre fibration $F\to E\to B$ is associated a long exact sequence in homotopy groups:
\[\cdots\to\pi_{n+1}B\to\pi_n F\to\pi_n E\to\pi_n B\to\pi_{n-1}F\to\cdots.\]

From the fibration $U_n \to U_{n+1} \to S^{2n+1}$ given by the standard action of $U_{n+1}$ on $S^{2n+1}$ one deduces that the stabilization map $U_n \to U_{n+1}$, which is given by adding a row and a column with zeros everywhere except for a 1 in the diagonal entry, induces isomorphisms on all $\pi_k$ for $k<2n$ and an epimorphism on $\pi_{2n}$. Indeed, $\pi_k S^{2n+1}=0$ for $k<2n+1$. The homotopy groups $\pi_k U := \pi_k U_n$ in the stable range $k<2n$ exhibit 2-fold periodicity and were computed by Bott \cite{B59} as follows:

\begin{center}
	\begin{tabular}{ c | c }
		$\mod(k,2)$ & $\pi_{k}U$\\\hline
		0 & $0$ \\ 
		1 & $\Z$\\   
	\end{tabular}%
	\end{center}

Similarly, from the fibration $O_n \to O_{n+1} \to S^n$ given by the standard action of $O_{n+1}$ on $S^n$ one deduces that the analogous stabilization map $O_n \to O_{n+1}$ induces isomorphisms on all $\pi_k$ for $k<n-1$ and an epimorphism on $\pi_{n-1}$. The homotopy groups $\pi_k O := \pi_k O_n$ in the stable range $k<n-1$ exhibit 8-fold periodicity and were also computed by Bott as follows:

\begin{center}
	\begin{tabular}{ c | c }
		$\mod(k,8)$ & $\pi_{k}O$\\\hline
		0 & $\Z/2$ \\ 
		1 & $\Z/2$\\   
		2 & $0$\\ 
		3 & $\Z$\\ 
		4 & $0$\\ 
		5 & $0$\\ 
		6 & $0$\\ 
		7 & $\Z$\\ 
	\end{tabular}%
	\end{center}

The Lagrangian Grassmannian $\Lambda_n$ admits a transitive action of $U_n$ with the stabilizer $O_n$, hence can be described as the homogeneous space $U_n/O_n$. By considering the long exact sequence in homotopy associated to the resulting fibration $O_n \to U_n \to \Lambda_n$, it follows from the above that the stabilization map $\Lambda_n \to \Lambda_{n+1}$, which is given by taking the direct sum in $\C^{n+1}=\C^n \oplus \C$ of a linear Lagrangian subspace of $\C^n$ and $\R \subset \C$, induces isomorphisms on all $\pi_k$ for $k<n$ and an epimorphism on $\pi_n$. 

The homotopy groups $\pi_k \Lambda := \pi_k \Lambda_n$ in the stable range $k<n$ exhibit 8-fold periodicity and were also computed by Bott, in fact they are just a shift of the stable homotopy groups $\pi_k O$ due to the homotopy equivalence $\Omega \Lambda \simeq \Z \times BO$. 

\begin{center}
	\begin{tabular}{ c | c }
		$\mod(k,8)$ & $\pi_{k}\Lambda$\\\hline
		0 & $0$ \\ 
		1 & $\Z$ \\ 
		2 & $\Z/2$ \\ 
		3 & $\Z/2$\\   
		4 & $0$\\ 
		5 & $\Z$\\ 
		6 & $0$\\ 
		7 & $0$\\ 
	\end{tabular}%
	\end{center}

For the purposes of this article we are interested not in the stable homotopy groups of $\Lambda_n$ but in the unstable group $\pi_n \Lambda_n$. Via the long exact sequence in homotopy of the fibration $O_n \to U_n \to \Lambda_n$ we may relate this group to the homotopy groups $\pi_n U_n$ and $\pi_{n-1}O_n$, the first of which is in the stable range but the second of which is not. The groups $\pi_n \Lambda_n$ and $\pi_{n-1}O_n$ are the first nonstable homotopy groups of $\Lambda_n$ and $O_n$ respectively.

These homotopy groups, though nonstable, are also understood. Not only do they surject onto the corresponding stable groups, but they exhibit a secondary form of 8-fold of periodicity, with three exceptions related to the parallelizability of $S^1$, $S^3$ and $S^7$. 

The computation of $\pi_{n-1}O_n$ is mostly straightforward, see \cite{S51}, but the non-parallelizability of $S^n$ for $n \neq 1,3,7$ \cite{BM58, K58, M58} plays an essential role. Here is the table for $\pi_{n-1}O_n$, where we remark that the indexing of $\pi_*$ is by $n-1$ instead of $n$ for future convenience when analyzing the sequence $\pi_n U_n \to \pi_n \Lambda_n \to \pi_{n-1}O_n$.

\begin{center}
	\begin{tabular}{ c | c }
		$\mod(n,8), \, \, n \geq 8$ & $\pi_{n-1}O_n$\\\hline
		0 & $\Z\oplus\Z$ \\ 
		1 & $\Z/2\oplus\Z/2$\\  
		2 & $\Z/2 \oplus\Z$\\
		3 & $\Z/2$ \\ 
		4 & $\Z\oplus\Z$ \\  
		5 & $\Z/2$ \\
		6 & $\Z$ \\ 
		7 & $\Z/2$ \\
	\end{tabular}%
\hspace{1in}
	\begin{tabular}{ l }
		 (small $n$)\\\hline
		$\pi_0O_1 = \Z/2$ \\ 
		$\pi_1O_2 = \Z$\\  
		$\pi_2O_3 = 0$\\
		$\pi_3O_4 = \Z\oplus\Z$\\ 
		$\pi_4O_5 = \Z/2$ \\  
		$\pi_5O_6 = \Z$\\
		$\pi_6O_7 = 0$ \\ 
		$\pi_7 O_8 = \Z \oplus \Z$
	\end{tabular}
\end{center} 

Briefly, to relate this table with that of the stable groups $\pi_{n-1}O$ one uses the fact that $\pi_{n-1}O_n \to \pi_{n-1}O$ is an epimorphism and $\ker(\pi_{n-1}O_n \to \pi_{n-1}O)= \text{im}(\pi_{n-1}S^{n-1} \to \pi_{n-1}O_n)$ is generated by the class of the tangent bundle $TS^{n}$, which has infinite order if $n$ is even, has order $2$ if $n$ is odd and not equal to $1,3,7$, and is trivial if $n=1,3$ or $7$.

The groups $\pi_{n}\Lambda_n$ were computed by Kachi in \cite{K78} and are given as follows:

\begin{center}
	\begin{tabular}{ c | c }
		$\mod(n,8), \, \, n \geq 8$ & $\pi_{n}\Lambda_n$\\\hline
		0 & $\Z$ \\ 
		1 & $\Z\oplus\Z/2$\\  
		2 & $\Z/2 \oplus\Z$\\
		3 & $\Z/2\oplus\Z/2$ \\ 
		4 & $\Z$ \\  
		5 & $\Z\oplus\Z/2$ \\
		6 & $\Z$ \\ 
		7 & $\Z/2$ \\
	\end{tabular}%
	\hspace{1in}
	\begin{tabular}{ l }
		(small $n$)\\\hline
		$\pi_0\Lambda_0 = 0$ \\ 
		$\pi_1\Lambda_1 = \Z$\\  
		$\pi_2\Lambda_2 = \Z$\\
		$\pi_3\Lambda_3 = \Z/4$\\ 
		$\pi_4\Lambda_4 = \Z$ \\  
		$\pi_5\Lambda_5 = \Z\oplus\Z/2$\\
		$\pi_6\Lambda_6 = \Z$ \\ 
		$ \pi_7\Lambda_7=\Z/2$ \\ 
	\end{tabular}
	\hspace{1in}
\end{center}

\begin{remark}
Strictly speaking the computation in \cite{K78} is for $\pi_n SU_n/SO_n$, however this group is isomorphic to $\pi_n U_n/O_n$ whenever $n >1$. This follows immediately from the long exact sequences in homotopy associated to the determinant fibrations $SU_n \to U_n \to U_1$ and $SO_n \to O_n \to O_1$.\end{remark}

Finally, $\ker(\pi_n \Lambda_n \to \pi_n \Lambda_{n+1})$ is given as follows:

\begin{center}
	\begin{tabular}{ c | c }
		$\mod(n,2), \, \, n > 1$ & $\ker(\pi_{n}\Lambda_n \to \pi_n \Lambda_{n+1})$\\\hline
		0 & $\Z$ \\ 
		1 & $\Z/2$\\  
	\end{tabular}%
	\end{center}
	
This table follows almost immediately from the previous ones and in any case is a consequence of the computation below. In almost all cases the subgroup $\ker(\pi_n \Lambda_n \to \pi_n \Lambda_{n+1})$ is a direct summand of $\pi_n \Lambda_n$ (with the other direct summand given by $\pi_n \Lambda_{n+1} \simeq \pi_n \Lambda$), but there are some exceptions in which it is given by:
\begin{itemize}
\item[($n$=1)] The trivial subgroup.
\item[($n$=2)] The index 2 subgroup $2\Z \subset \Z \simeq \pi_2\Lambda_2$.
\item[($n$=3)]The cyclic subgroup of order 2 in $\pi_3\Lambda_3 \simeq \Z/4$.
\end{itemize}

\begin{remark} Note that in all cases $\ker(\pi_n \Lambda_n \to \pi_n \Lambda_{n+1}) \subset \pi_n\Lambda_n$ is cyclic and we will give an explicit generator. \end{remark}

\subsection{A homotopical lemma}

The following lemma is the main homotopical input needed to prove our main theorem in the non-exceptional dimensions $n \neq 1,3,7$.

\begin{lemma}\label{lem: key lem}
Let $n \neq 1,3,7$. Then $\ker( \pi_n \Lambda_n \to \pi_{n-1} O_n ) \cap \ker( \pi_n \Lambda_n \to \pi_n \Lambda_{n+1}) =0$.
\end{lemma}

\begin{proof} Let $\beta \in \ker( \pi_n \Lambda_n \to \pi_{n-1} O_n ) \cap \ker( \pi_n \Lambda_n \to \pi_n \Lambda_{n+1})$. We proceed by cases to show that $\beta=0$.

\subsubsection{The case $n \equiv 0$ mod $(2)$}

If $n$ is even then $\pi_n U_n = 0$, hence the map $\pi_n \Lambda_n \to \pi_{n-1} O_n$ is a monomorphism and so $\beta \in \ker( \pi_n \Lambda_n \to \pi_{n-1} O_n)$ is necessarily zero.

\subsubsection{The case $n \equiv 7$ mod $(8)$, $n \geq 15$}

In this case the map $\pi_n O_n \to \pi_n U_n$ is an epimorphism by commutativity of the diagram
\[ \xymatrix{
	\pi_n O_n \ar[d] \ar[r] & \pi_n U_n \ar[d]\\
	\pi_n O  \ar[r]           &  \pi_{n} U }\]
Indeed, for $n \equiv 7$ mod $(8)$, $n \geq 15$, we note:
\begin{enumerate}
\item $\pi_n O_n \to \pi_n O$ an epimorphism as shown by Kervaire \cite{K60},
\item $\pi_n O \to \pi_nU$ is an isomorphism since $\pi_{8k+6}\Lambda = \pi_{8k+7} \Lambda=0$,
\item$\pi_n U_n \to \pi_n U$ is also an isomorphism since $\pi_nU_n$ is in the stable range,
\end{enumerate}
from which the conclusion follows. Hence $\pi_n U_n \to \pi_n \Lambda_n$ is the zero map, so $\pi_n \Lambda_n \to \pi_{n-1} O_n$ is an a monomorphism and we can argue as in the previous case.

\subsubsection{The case $n \equiv 1,5$ mod $(8)$}

In this case $\pi_n O_{n+1} \to \pi_n U_{n+1}$ is the zero map since $\pi_n O_{n+1}$ is isomorphic to $0$ or $\Z/2$ for $n$ congruent to 1 or 5 respectively while $\pi_n U_{n+1}$ is isomorphic to $\Z$. Hence the map $\pi_n U_{n+1}  \to \pi_n \Lambda_{n+1}$ is a monomorphism. We can therefore argue as follows. 

Let $\beta \in \ker( \pi_n \Lambda_n \to \pi_{n-1} O_n ) \cap \ker( \pi_n \Lambda_n \to \pi_n \Lambda_{n+1})$. Since $ \ker( \pi_n \Lambda_n \to \pi_{n-1} O_n ) = \text{im}( \pi_n U_n \to \pi_n \Lambda_n)$ we can lift $\beta$ to an element $\wh \beta \in \pi_n U_n$.  By commutativity of the diagram   
\[ \xymatrix{
	\pi_n U_n \ar[d] \ar[r] & \pi_n U_{n+1} \ar[d]\\
	\pi_n \Lambda_n  \ar[r]           &  \pi_{n} \Lambda_{n+1}} \]
it follows that the image $s(\wh \beta)$ of $\wh \beta$ under the stabilization map $\pi_n U_n \to \pi_n U_{n+1}$ is in the kernel of the map $\pi_n U_{n+1} \to \pi_n \Lambda_{n+1}$, Since $\pi_n U_{n+1} \to \pi_n \Lambda_{n+1}$ is a monomorphism, this implies $s(\wh \beta)=0$. But $\pi_n U_n \to \pi_n U_{n+1}$ is an isomorphism, so we must also have $\wh \beta=0$ and hence we conclude $\beta = 0$.

\subsubsection{The case $n \equiv 3$ mod $(8)$, $n \geq 11$}

In this case we have 
\[ \pi_n \Lambda_n = \Z/2 \oplus \Z/2, \quad \pi_n \Lambda_{n+1} = \Z/2 , \quad \pi_{n-1} O_n = \Z/2, \quad \pi_{n-1} O_{n+1} = 0. \]
Hence $\pi_n U_{n+1} \to \pi_n \Lambda_{n+1}$ is the unique nontrivial map $\Z \to \Z/2$ with the kernel $2 \Z \subset \Z$. 

Let $\beta \in \ker( \pi_n \Lambda_n \to \pi_{n-1} O_n ) \cap \ker( \pi_n \Lambda_n \to \pi_n \Lambda_{n+1})$. As in the previous case, we may choose $\wh \beta \in \pi_n U_n$ a lift of $\beta$, and the image $s(\wh \beta)$ of $\wh \beta$ under the stabilization map $\pi_n U_n \to \pi_n U_{n+1}$ is in the kernel of $\pi_n U_{n+1} \to \pi_n \Lambda_{n+1}$. It follows that $s(\wh \beta)$ is divisible by 2 in $\pi_n U_{n+1}$.

Since $\pi_n U_n \to \pi_n U_{n+1}$ is an isomorphism, we deduce that $\wh \beta$ is also divisible by 2, hence the same is true of $\beta$. But $\pi_n \Lambda_n$ is 2-torsion, so we conclude $\beta = 0$.

Having exhausted all cases, the proof is complete. \end{proof}

\subsection{Proof of the main theorem for $n \neq 1,3,7$}

Assume $n > 1$ in what follows.

\subsubsection{An Euler number computation}

Recall that a formal fold $(\Sigma,v)$ in $\R^n$ determines an element $\alpha(\Sigma,v)$ in $\ker(\pi_n \Lambda_n \to \Lambda_{n+1})$. The image of $\alpha(\Sigma,v)$ in $\pi_{n-1}O_n$ lies in the kernel of the map $\pi_{n-1}O_n \to \pi_{n-1} O_{n+1}$ by commutativity of the diagram
\[ \xymatrix{
	\pi_n \Lambda_n \ar[d] \ar[r] & \pi_n \Lambda_{n+1} \ar[d]\\
	\pi_{n-1}O_n  \ar[r]           &  \pi_{n-1} O_{n+1}   } \]
	
This is just a diagram chasing way of saying that since $\alpha(\Sigma,v)$ is a stably trivial Lagrangian distribution, in particular the underlying real vector bundle is stably trivial. It turns out that all stably trivial real vector bundles arise in this way:

\begin{lemma}\label{ref: poincare-hopf}
The images of the elements $\alpha(\Sigma,v)$ in $\pi_{n-1}O_{n}$ generate the subgroup $\ker( \pi_{n-1}O_n \to \pi_{n-1}O_{n+1})$.
\end{lemma}

\begin{proof}
Consider first the case where $n$ is even. The subgroup $\ker( \pi_{n-1}O_n \to \pi_{n-1}O_{n+1}) = \text{im}( \pi_{n}S^{n} \to \pi_{n-1}O_n)$ is infinite cyclic and generated by $TS^n$, so it is enough to show that the Euler numbers of the real vector bundles underlying the distributions $\alpha(\Sigma,v)$ can realize any even integer. Let $\Omega \subset \R^n$ be a compact domain with smooth boundary. Set $\Sigma = \p \Omega$ and let $v$ be the outward normal to $\Omega$. A straightforward application of the Poincar\'e-Hopf index theorem shows that the Euler number of $(\Sigma, v)$ is equal to $2\chi(\Omega)$. Since $n > 1$, we can arrange for $\chi(\Omega)$ to take any integer value, which completes the proof.

Consider next the case where $n$ is odd. For $n=1,3,7$ the group $\pi_{n-1}O_n$ is trivial so there is nothing to prove. For $n \neq 1,3,7$ the group $\pi_{n-1}O_n$ has a single stably trivial element, which is the class of $TS^n$. By direct inspection this element is equal to the image of $\alpha(S^{n-1},n_{D^n})$ in $\pi_{n-1}O_n$, where $n_{D^n}$ is the outward normal to the unit disk $D^n=\{\|x\|\leq 1: \, \, x \in \R^n \}$. 
\end{proof}

\begin{corollary}\label{cor: iso}
For $n \neq 1,3,7$ the map $\ker(\pi_n \Lambda_n \to \pi_n \Lambda_{n+1}) \to \ker(\pi_{n-1}O_n \to \pi_{n-1}O_{n+1})$ is an isomorphism.
\end{corollary}

\begin{proof}
Injectivity is given by Lemma \ref{lem: key lem} while surjectivity is given by Lemma \ref{ref: poincare-hopf}.
\end{proof}

\begin{remark}
When $n=3$ we have $\ker(\pi_3 \Lambda_3 \to \pi_3 \Lambda_{4}) \simeq 2 \cdot \Z/4 \subset \Z/4 \simeq \pi_3 \Lambda_3$ while $\pi_2O_3=0$. When $n=7$ we have  $\ker(\pi_7 \Lambda_7 \to \pi_7 \Lambda_{8}) = \pi_7 \Lambda_7\simeq \Z/2$ while $\pi_6O_7=0$. \end{remark}

\subsubsection{Conclusion of the proof}

We are now ready to prove Theorem \ref{thm: main} in the non-exceptional dimensions, which we recall states that every element of $\ker( \pi_n \Lambda_n \to \pi_n \Lambda_{n+1})$ admits a representative of the form $\alpha(\Sigma,v)$ for some formal fold $(\Sigma,v)$ in $\R^n$.

\begin{proof}[Proof of Theorem \ref{thm: main} for $n \neq 1,3,7$]  
 Let $\gamma \in \ker(\pi_n \Lambda_n \to \Lambda_{n+1})$. By Lemma \ref{ref: poincare-hopf} there exists a formal fold $(\Sigma,v)$ such that the images of $\gamma$ and $\alpha(\Sigma,v)$ in $\pi_{n-1}O_n$ are equal. But Corollary \ref{cor: iso} implies that $\gamma$ and $\alpha(\Sigma,v)$ are in fact equal in $\pi_n \Lambda_n$. \end{proof}

\begin{remark}\label{rem: gen}

It follows that for any $n \neq 1,3,7$ the subgroup $\ker(\pi_n \Lambda_n \to \pi_n \Lambda_{n+1})$ is cyclic with a generator given by $\alpha(S^{n-1},n_{D^n})$, where $n_{D^n}$ is the outward normal the unit disk $D^n=\{ \|x\| \leq 1: \, \, x \in \R^n \}$. As we will see below the same is true for the exceptional dimensions $n=1,3,7$. The element $\alpha(S^{n-1},n_{D^n})$ is 2-torsion for $n$ odd. For $n$ even it is not and we can obtain representatives for its multiples as follows. Given $k>0$, the element of $\pi_n \Lambda_n$ given by $k$ times $\alpha(S^{n-1}, n_{D^n})$ is equal to $\alpha(\p \Omega_k, n_{\Omega_k})$, where $\Omega_k$ is the disjoint union of $k$ disks in $\R^n$ and $n_{\Omega_k}$ is the outward normal. More generally, if $\Omega \subset \R^n$ is any domain of Euler characteristic $k$ then $\alpha(\p \Omega, n_\Omega)$ is a representative for $k\alpha(S^{n-1}, n_{D^n})$, where $n_\Omega$ is the outward normal to $\Omega$. Similarly, one can obtain a representative for $-k\alpha(S^{n-1},n_{D^n})$ by taking $\alpha(\p \Omega,n_\Omega)$, where $\Omega \subset \R^n$ is any domain with Euler characteristic $-k$ and $n_\Omega$ is the outward normal to $\Omega$.

\end{remark}

\section{The exceptional cases}\label{sec: excep}

\subsection{Complex trivializations}  To tackle the exceptional dimensions $n=3,7$ we will make use of an explicit complex trivialization of $T(T^*S^n)|_{S^n}$ together with a certain property of stable triviality satisfied by the trivialization. This trivialization is defined for all $n \geq 1$, but for $n=3,7$ we will further examine its interaction with the trivializations coming from quaternionic and octonionic geometry (which are not stably trivial), leading to a proof of Theorem \ref{thm: main} in those dimensions.

\subsubsection{The isomorphism $T(T^*S^n)|_{S^n} \simeq S^n \times \C^n$ } 
We will make use an explicit isomorphism of symplectic vector bundles $T(T^*S^n)|_{S^n} \simeq S^n \times \C^n$, which is defined as follows.

Let $e$ be the first unit vector in $\R^{n+1}$, and let $\theta=\theta(g)$ measure the angle of a vector $g\in S^n$ away from $e$; importantly, $\cos\theta=\langle e,g\rangle$. Each level set $\theta=\theta_0$ (other than $\theta_0=0,\pi$) is isometric to the scaled sphere $\sin(\theta)\;S^{n-1}$ by the mapping $g\mapsto g-\cos(\theta_0)e$. For any point $g$ in the level set $\theta=\theta_0$ and each vector $v\in\{e\}^\perp\subset\R^{n+1}$, define the coordinate $\theta_v=\theta_v(g)$ to be the angle between $g-\cos(\theta)e$ and $v$. These coordinates are well-defined except at the poles $\pm e$. In particular, fix an orthonormal basis $e_1,...,e_n$ of $\{e\}^\perp$, and define $\theta_i=\theta_{e_i}$.
	
	Now define the (discontinuous) vector fields $X=-\partial_\theta$ and $X_v=-\partial_{\theta_v}$, and write $X_i=-\partial_{\theta_i}$. Writing $J$ for the standard\footnote{That is, $J$ is compatible with the round metric and the canonical symplectic structure.} almost complex structure on $T(T^*S^n)|_{S^n}$, we define a complex trivialization of $T(T^*S^n)|_{S^n}$ by
	\begin{equation}\label{coolframe}
		E_i=-\cos(\theta_i)e^{J\theta}X+\sin(\theta_i)X_i,
	\end{equation}
	
\begin{definition}\label{def:isomorphism}
	For any $n\geq 1$, define the bundle map $F:T(T^*S^n)|_{S^n}\to S^n \times \C^n$ to be the one taking $E_i|_g$ to $(g,e_i)$ at each point $g\in S^n$.
\end{definition}

\begin{lemma}\label{lemma:isomorphism}
	For any $n\geq 1$, the map $F:T(T^*S^n)|_{S^n}\to S^n \times \C^n$ is a complex vector bundle isomorphism.
\end{lemma}
\begin{proof}
	The lemma boils down to showing that the maps $F|_g:T_g(T^*S^n)\to \{g\}\times\C^n$ are linear isomorphisms and vary continuously with $g\in S^n$; in turn, this follows from showing that $\{E_i\}$ is a continuous frame.
	
	It is clear from the expressions (\ref{coolframe}) that the sections $E_i$ are continuous and well-defined, so it remains to be seen that they are complex-linearly independent.
	
	Let $a,b\in\{e\}^\perp\subset\R^{n+1}$, and set $v=a+Jb=\sum (a^i+Jb^i)e_i$. We show that the complex linear combination $\sum (a^i+Jb^i)E_i$ results in $E_v:=E_a+JE_b$, where we define
	\[E_a=\|a\|\left(-\cos(\theta_a)e^{J\theta}X+\sin(\theta_a)X_a\right)\]
	and similarly for $E_b$. Indeed, for any $g\in S^n$,
	\begin{align}\label{equation:anglething}
	\begin{split}
		\sum (a^i+Jb^i)E_i|_g& =-\left[\sum(a^i+Jb^i)\cos(\theta_i)\right]e^{J\theta}X\\
		&\qquad+\sum(a^i+Jb^i)\sin(\theta_i)X_i.\\
		& =-\left[\langle a,g\rangle + J\langle a,g\rangle\right]e^{J\theta}X\\
		&\qquad+\sum(a^i+Jb^i)\sin(\theta_i)X_i.\\
		& =-\left[\|a\|\cos(\theta_a) + J\|b\|\cos(\theta_b)\right]e^{J\theta}X\\
		&\qquad+\sum(a^i+Jb^i)\sin(\theta_i)X_i.
	\end{split}
	\end{align}
	To deal with the second term above, rewrite
	\[X_u|_{g} = \left(\sin\theta_u\right)^{-1}\left(u-\cos(\theta_u)\cdot\frac{g-\cos(\theta)e}{\sin(\theta)}\right).\]
	Then we find
	\begin{align*}
		\sum(a^i+Jb^i)\sin(\theta_i)X_i &= \sum(a^i+Jb^i)\left(e_i-\cos(\theta_i)\cdot\frac{g-\cos(\theta)e}{\sin(\theta)}\right)\\
		&=a+Jb -\left(\langle a,g\rangle + J\langle b,g\rangle\right)\cdot\frac{g-\cos(\theta)e}{\sin(\theta)}\\
		&=a+Jb -\left(\|a\|\cos(\theta_a) + J\|b\|\cos(\theta_b)\right)\cdot\frac{g-\cos(\theta)e}{\sin(\theta)}\\
		&=\|a\|\sin(\theta_a)X_a+\|b\|\sin(\theta_b)X_b.
	\end{align*}
	Putting these two elements together implies
	\[\sum (a^i+Jb^i)E_i=E_a+JE_b=E_v.\]
	Since $E_v$ is nonzero for $v$ nonzero, this proves our result.
\end{proof}

\begin{remark}
Consider the Lagrangian distribution $\gamma \subset T(T^*S^n)|_{S^n}$ which is the preimage by $F$ of $S^n \times i\R^n$. It is straightforward to verify that $\gamma$ has fold type tangencies with $S^n$ along the equator $S^{n-1} =\{ x_0 =0\} \cap S^n$ and is transverse to $S^n$ everywhere else. 
\end{remark}

\subsubsection{Stable triviality of the frame}

We will also need the fact that the frame defined above is {\em stably trivial}, in the following sense. By stabilizing once, the vector bundle isomorphism $F:T(T^*S^n)|_{S^n}\to S^n\times\C^n$ extends to an isomorphism $ \wh F:T(T^*S^n)|_{S^n}\times\C\to S^n\times\C^{n+1}$. Identifying the extra factor of $\C$ as the complexification of the normal direction to the sphere and using the trivialization $T(T^*\R^{n+1})|_{S^n} \simeq S^n \times \C^{n+1}$ we may rewrite this as a map $$ \wh F:S^n\times\C^{n+1}\to S^n\times \C^{n+1}$$ which is a lift of the identity map $S^n \to S^n$ by fibrewise linear isomorphisms $\C^{n+1} \to \C^{n+1}$. 

In fact, these are unitary transformations, as can be verified using the explicit formulas provided by Definition \ref{def:isomorphism}.

\begin{lemma}\label{lemma:stabilize}
	The map $ \wh F:S^n\times\C^{n+1}\to S^n\times \C^{n+1}$ is trivial as an element of $\pi_nU_{n+1}$. 
\end{lemma}

\begin{remark} We note:
\begin{enumerate}
\item As a basepoint of $S^n$ we take the point $e$ where the frame $(E_i)_i$ agrees with the frame $e_i$, and as a basepoint of $U_n$ we take the identity matrix. 
\item It is sufficient to prove the triviality of $\wh F$ as an element of $\pi_n \GL(\C^{n+1})$ since the inclusion $U_n \subset \GL(\C^{n+1})$ is a homotopy equivalence.	
\end{enumerate}
\end{remark}

\begin{proof}
We continue in the language of the proof of Lemma \ref{lemma:isomorphism}.
	
	By stabilizing, we introduce a new vector field $E_0$ to our frame, everywhere orthogonal to $E_{i\neq 0}$. We can view this as an outward normal field to $S^n$. In short, $E_0\equiv\partial_R$, where $R$ is the outward radial coordinate (the norm in $\R^{n+1}$). In this setting, $\wh F$ takes the form 
	\[\wh F:S^n\times\C^{n+1}\to S^n\times\C^{n+1},\qquad E_i|_g\mapsto(g,e_i),\]
	with $\{e=e_0,e_1,...,e_n\}$ an orthonormal basis of $\C^{n+1}$. Our lemma thus boils down to the following claim: the frame $\{E_i\}$ is homotopic to the trivial frame $\{e_i\}$ through maps $S^n\to \GL(\C^{n+1})$. Indeed, if this is the case, then we can pre-compose the map $\wh F:E_i\mapsto e_i$ with this homotopy to perturb $\wh F$ itself continuously to the identity map $e_i\mapsto e_i$.
	
	We prove this by supplying a sequence of perturbations bringing $E_i$ to $e_i$; it of course is crucial (and we will prove this along the way) that the perturbations of the frames are complex-linearly independent at each point in time. 
	
	To begin, we apply two continuous homotopies from $E_i=E'_i(0)$ to $E'_i(1)$. 
	\begin{align*}
		E_i'(t)&=E_i +t\cos(\theta_i)\sin(\theta)\partial_R,\\
		E_0'(t) &= e^{-J\theta t}E_0 + t\sin(\theta)X_0,
	\end{align*}
	where we let $t$ go from 0 to 1 (and we exclude 0 from the index $i$). We can extend these perturbations to general $E_v$, where $v=a+Jb=\sum_{i>0} (a^ie_i+b^iJe_i)$, $a,b \in \{e\}^\perp \subset \R^{n+1}$ using the formula
	\[E_v'(t)=E_v +t\left(\|a\|\cos(\theta_a)+J\|b\|\cos(\theta_b)\right)\sin(\theta)\partial_R.\]
	The resulting transformation $E_v\mapsto E'_v$ is a complex-linear isomorphism, as we see from the following calculation:
	\begin{align*}
		\sum\nolimits_i \left(a^iE'_{i}(t)+b^iJE'_{j}(t)\right)\big|_g&=\sum\nolimits_i \left(a^iE_{i}+b^iJE_{j}\right)\\
		&\qquad +t\sum\nolimits_i \left(a^i\cos(\theta_i)+b^iJ\cos(\theta_i)\right)\sin(\theta)\partial_R,\\\
		&=E_{v} +t\sum\nolimits_i \left((a^i+b^iJ)\langle x,\mathbf{e}_i\rangle\right)\sin(\theta)\partial_R,\\
		&=E_{v} +t\left(\langle g,a\rangle+J\langle g,b\rangle\right)\sin(\theta)\partial_R,\\
		&=E_{v} +t\left(\|a\|\cos(\theta_a)+J\|b\|\cos(\theta_b)\right)\sin(\theta)\partial_R,
	\end{align*}
	In particular, the $E'_{i}(t)$ remain linearly independent throughout the perturbation. Furthermore, this means that if $E_0'(t)$ were in the span of $\{E'_i(t)\}$, then we would have $E_0'(t)=E'_v(t)$ for some nonzero $v$. We can see that this would require $t=1$; otherwise, the projection of $E'_0$ to $T(T^*S^n)|_{S^n}$ has norm strictly larger than the projection of $E'_0$ to $\operatorname{span}(\partial_R)$, while the opposite is true of $E'_v$.  
	
	For the same reason, we see that this can only happen when $\sin(\theta)=1$ and $\cos(\theta_v)=\pm 1$, and thus only when $v=\beta a$ for some nonzero $\beta \in \C$ and some vector $a\in \{e\}^\perp \subset \R^{n+1}$. Indeed, we could not satisfy $\cos(\theta_a),\cos(\theta_b)=\pm 1$ for $a$ not a multiple of $b$, so we could not have $v=a+Jb$ for such vectors. Knowing this, we have
	\[E'_0 = -J\partial_R + X_0, \qquad E'_v = \mp J\beta X_0\pm \beta\partial_R,\]
	for some nonzero $\beta\in\C$. These are indeed independent; within the subspace spanned by $X_0$ and $\partial_R$, these two vectors give a determinant of $\pm 2\beta\neq 0$.
	
	Finally, we make the two perturbations
	\begin{align*}
		E_i''(s)&=E'_i(1) +sJ\cos(\theta_i)\sin(\theta)X_0,\\
		E_0''(s) &=E'_0(1)+ sJ\sin(\theta)\partial_R,
	\end{align*}
	which extend as before to general $E'_v$.
	For clarity, here are the closed-form expressions of $E_0''$ and $E_v''$, for $v \in \{e\}^\perp \subset \R^{n+1}$.
	\begin{align*}
		E_v''(s)&=\|v\|\left[\sin(\theta_i)X_v-\cos(\theta_v)\cos(\theta)X_0+\cos(\theta_v)\sin(\theta)\partial_R+(s-1)J\cos(\theta_v)\sin(\theta)X_0\right],\\
		E_0''(s) &=\cos(\theta)\partial_R+ \sin(\theta)X_0+(s-1)J\sin(\theta)\partial_R.
	\end{align*}
	Since $v\mapsto E_v''(s)$ is a complex-linear, the vector fields $\{E''_0(s),E''_i(s)\}_i$ can only be linearly dependent if $E_0''(s)=E_v''(s)$ for some $v\neq 0$. Suppose $v=a+Jb$, for $a,b\in \{e\}^\perp \subset \R^{n+1}$. From the above expressions, we can only satisfy $E_0''(s)=E_a''(s)+JE_b''(s)$ at points $g\in S^n$ where the components along $X_a$ and $X_b$ vanish; this requires $\sin(\theta_a)=\sin(\theta_b)=0$, further implying that $\cos(\theta_a),\cos(\theta_b)=\pm 1$, that $a$ is a multiple of $b$, and thus that $E_v''(s)=\beta E_a''(s)$ for some nonzero $\beta\in\C$. In this case, the determinant of $E_0''(s),E_v''(s)$ within the subspace $\operatorname{span}(X_0,\partial_R)$ is
	\[\beta\|v\|\det\left(\begin{smallmatrix}
		\sin\theta & \mp\cos\theta\pm (s-1)J\sin\theta\\
		\cos\theta+(s-1)J\sin\theta& \pm\sin(\theta)
	\end{smallmatrix}\right)=\pm \beta\|v\|\left(1+(s-1)^2\sin^2\theta\right)\neq 0.\]
	Thus, the frame remains linearly independent over the full perturbation. Finally, it is clear from the above expressions that $E''_v(1)\equiv v$ and $E''_0(1)\equiv e$, which proves the lemma.
\end{proof}

\subsection{Proof of the main theorem in the exceptional cases}

It remains to prove Theorem \ref{thm: main} in the cases $n=1,3,7$. The case $n=1$ is trivial and will not be discussed further.

\subsubsection{The case $n=3$}
Identify $S^3\subset\R^4\simeq\mathbf{H}$ with the set of unit quaternions, giving it the structure of a Lie group. As such, we can recover an orthonormal trivialization $E_1,E_2,E_3$ of left-invariant vector fields in $TS^3$ by extending the basis $\{i,j,k\}$ of $T_{\mathbf{1}}S^3$.

This gives rise to a complex trivialization of $T(T^*S^3)|_{S^3}$ (equipped with the unique almost complex structure compatible with the symplectic form and the round metric) and thus an isomorphism $$G:T(T^*S^3)|_{S^3}\isomto S^3\times\C^3$$ distinct from the isomorphism $F$ considered in Lemma \ref{lemma:isomorphism}. 

Thus we obtain an element $\beta:=iF\circ G^{-1}\in\pi_3U_3$.

\begin{lemma}\label{lemma: folds}
The image $\alpha$ of $\beta$ in $\pi_3 \Lambda_3$ is equal to the element $\alpha(S^2,n_{D^3})$.
\end{lemma}
\begin{proof} We argue as follows. On the one hand, $G^{-1}$ maps $S^3\times\R^3$ to $TS^3$ and $S^3\times i\R^3$ to the vertical distribution $\nu$ of $T(T^*S^3)|_{S^3}$, by construction. On the other, $iF(TS^3)$ is transverse to $S^3\times \R^3$ everywhere away from the equator $\{\mathbf{1}\in\mathbf{H}\}^\perp\cap S^3$, but has fold tangencies with $S^3\times\R^3$ along that equator. Thus, $\alpha$ has folds along that same equator, but no other singularities anywhere else. \end{proof}

\begin{figure}[h]
\includegraphics[scale=0.5]{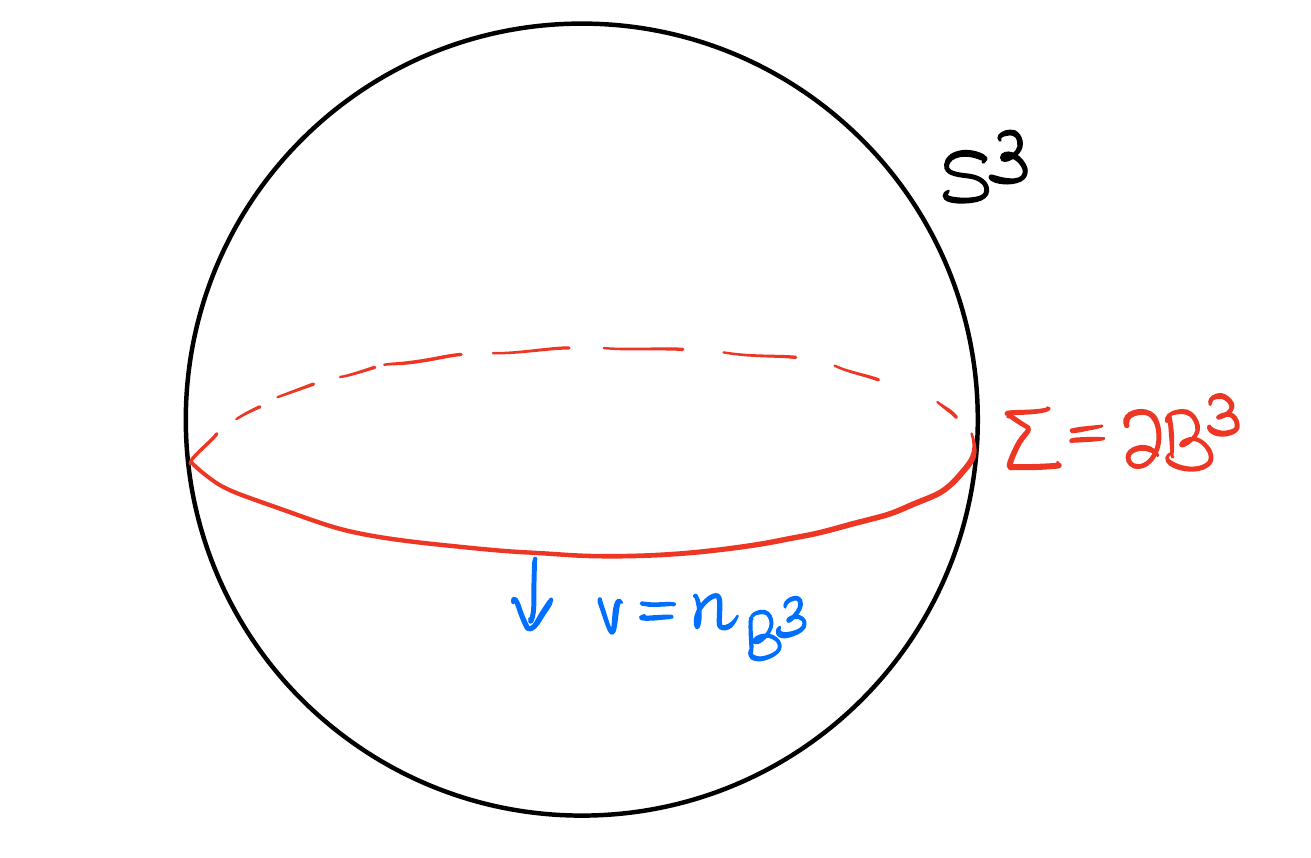}
\caption{The element $\alpha(S^2,n_{B^3}) \in \pi_3 \Lambda_3$.}
\label{element}
\end{figure}

Next, consider the element $\eta\in\pi_3O_4$ given by left-quaternion multiplication---that is, for $g\in S^3\subset\mathbf{H}$ and $h\in\R^4\simeq\mathbf{H}$, we have $\eta(g)\cdot h=gh\in\R^4$. We claim that the image of $\eta$ in $\pi_3U_4 \simeq \Z$ is equal to twice a generator, which can be seen as follows. It is well known that the image of $\eta$ in $\pi_3O$ is a generator \cite{AH61}. In particular, its image in the stable group $\pi_3O_5$ must be a generator. Consider the map $\pi_3O_5\to\pi_3U_5$ in the following diagram:
\begin{equation}\label{diagram:quatstabilization} \xymatrix{
		\pi_3O_5\ar[r]& \pi_3U_5\ar[r] & \pi_3\Lambda_5\ar[r]& \pi_2O_5\\
		\pi_3O_4\ar[u]\ar[r]&  \pi_3U_4  \ar[u]&   }
\end{equation}
Since $\pi_3\Lambda_5\simeq\Z/2$ and $\pi_2O_5\simeq 0$, the map $\pi_3O_5\to \pi_3U_5$, which is $\Z \to \Z$, sends a generator to twice a generator. Since $\pi_3U_4 \to \pi_3U_5$ is an isomorphism, it follows that the image of $\eta$ in $\pi_3U_4$ is equal to twice a generator, as claimed.

\begin{proof}[Proof of Theorem \ref{thm: main} in the case $n=3$]

Recall that $\pi_3\Lambda_3\simeq\Z/4$ and $\pi_3 \Lambda_4 \simeq \Z/2$, hence the stabilization map $\pi_3 \Lambda_3 \to \pi_3 \Lambda_4$  is the unique epimorphism $\Z/4 \to \Z/2$ whose kernel is the single element $2\in\Z/4$ corresponding to twice a generator.

Hence in view of Lemma \ref{lemma: folds} it suffices to prove that $ \alpha \in \pi_3 \Lambda_3$ is equal to twice a generator. Since the map $\pi_3 U_3 \to \pi_3 \Lambda_3$ sends a generator to a generator, it also suffices to prove that $\beta \in \pi_3 U_3$ is equal to twice a generator. Finally, since $\pi_3 U_3 \to \pi_3 U_4$ is an isomorphism and we know the image of the element $\eta$ under the map $\pi_3 O_4 \to \pi_3U_4$ to equal twice a generator, it suffices to show that the images of $\beta$ and $\eta$ in $\pi_3U_4$ are identical.

In stabilizing $\beta$ to an element $\wh \beta \in \pi_3U_4$, consider the additional $4^{th}$ unit vector field $E_0$ as the complexification of the outward unit normal to the sphere. Considering $S^3\times\C^4$ as a restriction of the tangent bundle of the complexified Lie group $\mathbf{H}\otimes\C$, note that $E_0$ is also left-invariant (similar to $E_1,E_2,E_3$). Extend $F$ and $G$ to complex vector bundle isomorphisms $\wh F, \wh G:S^3\times\C^4\to S^3\times\C^4$ by $\wh F(E_0|_g)=\wh G(E_0|_g)=(g,e_0)$. Since $E_0$ is left-invariant in the sense mentioned above, $\wh G^{-1}$ is simply the complexified quaternion multiplication map, i.e. the image of $\eta$ in $\pi_3U_4$. 

That $i\wh F$ is homotopic to the identity follows from Lemma \ref{lemma:stabilize}; thus, $\wh \beta=i\wh F\circ \wh G^{-1}$ is equal to the image of $\eta$ in $\pi_3U_4$, and we are done. \end{proof}

\begin{remark}\label{rem: alt}
We note that in the case $n=3$ one may alternatively argue in the following way. By chasing the diagram
\[ \xymatrix{
	\pi_3 U_2 \ar[d] \ar[r] & \pi_3 U_3 \ar[d]\\
	\pi_3 \Lambda_2  \ar[r]           &  \pi_3 \Lambda_3   } \]
it follows that a generator of $\pi_3 \Lambda_3$ is given by the stabilization of the image under $\pi_3 U_2 \to \pi_3 \Lambda_2$ of a generator of $\pi_3 U_2$. But from the determinant fibration $SU_2 \to U_2 \to U_1$ we see that a generator of $\pi_3 U_2$ is given by the image of a generator of $\pi_3 SU_2$, which we can take to be the identity map under the standard identification $S^3 \simeq SU_2$ given by

\[(z,w) \mapsto
\begin{bmatrix}
z     & - \overline{w}  \\
  w      & \overline{z}
\end{bmatrix}  , \qquad (z,w) \in S^3 \subset \C^2 .\]

One may explicitly compute the tangencies of the resulting map $S^3 \to \Lambda_2$ with a suitable Lagrangian plane to be $\Sigma^2$-nonsingular, and more precisely to consist of a $\Sigma^1$ locus on a torus which is the boundary of a standard genus 1 handlebody in $\R^3$, with $\Sigma^{110}$ pleats on a $(1,1)$ curve on the torus and no $\Sigma^{1110}$ points, see Figure \ref{11curve}. 

It is then an entertaining exercise in Entov's surgery of singularities \cite{En97} to show that the disjoint union of two copies of this chain of singularities can be surgered into a sphere of $\Sigma^{10}$ folds, i.e. into the element $\alpha(S^2,n_{B^3})$. Hence the generator of $\ker(\pi_3 \Lambda_3 \to \pi_3 \Lambda_4)$, which is equal to twice a generator of $\pi_3\Lambda_3$, is represented by $\alpha(S^2,n_{B^3})$. We know of no analogous explicit argument in the case $n=7$, which we discuss next.
\end{remark}

\subsubsection{The case $n=7$}
Just as in the quaternionic case earlier, we identify $S^7\subset \R^8\simeq\mathbf{O}$ with the set of unit octonions. Let $e=\mathbf{1}\in\mathbf{O}$ be the first unit vector, and---for any unit octonion $g\in S^7$---define the left-multiplication map
\[L_g:T_eS^7\to T_gS^7,\qquad (e,v)\mapsto(g,g\cdot v),\]
where we view $v\in T_eS^7\simeq\{x\in\R^8\simeq\mathbf{O}\;|\;\langle x,e\rangle=0\}$ as an octonion itself. This yields a trivialization $E_1,...,E_7$ of $TS^7$, where $E_i|_e=e_i$ is the $i^{th}$ unit vector in $\R^8\supset T_eS^7$ and $E_i|_{g}=L_gE_i|_{e}$. 

This gives rise to a complex trivialization of $T(T^*S^7)|_{S^7}$ (equipped with the unique almost complex structure compatible with the symplectic form and the round metric), and thus an isomorphism $$G:T(T^*S^7)|_{S^7}\isomto S^7\times\C^7$$ distinct from the isomorphism $F$ considered in Lemma \ref{lemma:isomorphism}. 

Thus we obtain an element $\beta:=iF\circ G^{-1}\in\pi_7U_7$.

\begin{lemma}\label{lemma: folds redux}
The image $\alpha$ of $\beta$ in $\pi_7 \Lambda_7$ is equal to the element $\alpha(S^6,n_{D^7})$.
\end{lemma}

\begin{proof} On the one hand, $G^{-1}$ maps $S^7\times\R^7$ to $TS^7$ and $S^7\times i\R^7$ to the vertical distribution $\nu$ of $T(T^*S^3)|_{S^3}$, by construction. On the other, $iF(TS^7)$ is transverse to $S^7\times \R^7$ everywhere away from the equator $\{e \}^\perp\cap S^7$, and has fold tangencies with $S^7 \times \R^7$ along that equator. Thus, $\alpha$ has folds along that same equator, but no other singularities anywhere else. \end{proof}

Next, consider the element $\eta\in\pi_7O_8$ given by left-octonion multiplication---that is, for $g\in S^7\simeq O_1(\mathbf{O})$ and $h\in\R^8\simeq\mathbf{O}$, we have $\eta(g)\cdot h=gh\in\R^8$. We claim that the image of $\eta$ in $\pi_7U_8$ is a generator, which can be seen as follows. It is well known that that the image of $\eta$ generates $\pi_7O$ \cite{AH61}; in particular, it generates $\pi_7O_9$. Consider the diagram
\begin{equation}\label{diagram:octstabilization} \xymatrix{
		\pi_7O_9\ar[r]& \pi_7U_9\ar[r] & \pi_7\Lambda_9\\
		\pi_7O_8\ar[u]\ar[r]&  \pi_7U_8 \ar[u] &   }
\end{equation}
From the fact that $\pi_7 \Lambda_9=0$, we see that $\pi_7O_9 \to \pi_7U_9$ is an isomorphism $\Z \to \Z$. Since $\pi_7U_8 \to \pi_7U_9$ is an isomorphism, it follows that the image of $\eta$ in $\pi_7U_8$ must be a generator, which establishes the claim.

\begin{proof}[Proof of Theorem \ref{thm: main} in the case $n=7$]
Recall that $\pi_7\Lambda_7\simeq\Z/2$ has only one nonzero element; it is stably trivial because $\pi_7\Lambda_8=0$.

In view of Lemma \ref{lemma: folds redux} we need to prove that $\alpha \in \pi_7\Lambda_7$ is this unique non-zero element. Since the map $\pi_7U_7 \to \pi_7\Lambda_7$ is the unique non-zero map $\Z \to \Z/2$, it suffices to show that $\beta \in \pi_7U_7$ is a generator. Finally, since the image of $\eta$ under the map $\pi_7O_8 \to \pi_7U_8$ is a generator and $\pi_7U_7 \to \pi_7U_8$ is an isomorphism, it suffices to show that the images of $\beta$ and $\eta$ in $\pi_7U_8$ are identical.
 
In stabilizing $\beta$ to an element $\wh \beta \in \pi_7U_8$, consider the additional $8^{th}$ unit vector field $E_0$ as the outward unit normal to the sphere. The multiplication map $L_g$ is diagonal along the radial coordinate, so we have $E_0|_g=L_gE_0|_e$, just as with $E_1,...,E_7$.

Extend $F$ and $G$ to (vector bundle) isomorphisms $\wh F, \wh G: S^7\times\C^8\to S^7\times\C^8$ by $\wh F(E_0|_g)=\wh G(E_0|_g)=(g,e_0)$. We can see that $\wh G^{-1}$ is simply the complexified octonion multiplication map, i.e. the image of $\eta$ in $\pi_7U_8$. Indeed, suppose $v=\sum v^ie_i$, and calculate
\[\wh G^{-1}(g,v)=\sum v^i\wh G^{-1}(g,e_i)=\sum v^iE_i|_g=\sum v^iL_ge_i=L_gv.\]
That $i\wh F$ is homotopic to the identity follows from Lemma \ref{lemma:stabilize}; thus, $\wh{\beta}=i \wh F\circ \wh G^{-1}$ is equal to the image of $\eta$ in $\pi_7U_8$, and the theorem is proved. \end{proof}

\section{Applications}\label{sec: app}

\subsection{Arborealization of Weinstein manifolds with a single handle}\label{sec: ap1}

We begin by briefly recalling some basic definitions of symplectic topology.

\begin{definition}
	A {\em Liouville domain} $(W, \lambda)$ consists of a compact manifold with boundary equipped with an exact symplectic form $\omega = d \lambda$ together with a choice of primitive $\lambda$ such that the vector field $Z$ which is $\omega$-dual to $\lambda$ is outwards pointing along $\p W$.
\end{definition}

\begin{definition}
	A {\em Weinstein domain} $(W, \lambda, \phi)$ consists of a Liouville domain $(W, \lambda)$ together with a Morse function $\phi : W \to \R$ which is Lyapunov for $Z$.
\end{definition}

The Lyapunov condition means that $Z$ is gradient-like for $\phi$. The skeleton of a Weinstein domain $(W, \lambda, \phi)$  is the union of the stable manifolds of the critical points of $\phi$, hence is a union of isotropic submanifolds. The skeleton is in general a quite singular object, but there is a particularly simple class of Lagrangian singularities introduced by Nadler in \cite{N13,N15} and further developed in \cite{St18,AGEN20a} called arboreal singularities.  That four dimensional Weinstein manifolds admit skeleta with arboreal singularities was proved in \cite{St18}. 

In arbitrary dimensions, it was proved in \cite{AGEN20b} that if $TW$ admits a global field of Lagrangian planes, then the Weinstein structure of $(W, \lambda, \phi)$ can be deformed so that the skeleton has arboreal singularities. The proof relies on the ridgification theorem \cite{AGEN19}, which builds on the h-principle for the simplification of caustics \cite{AG18b} but is somewhat more subtle and has a greater range of applicability. 

As a corollary of our main result Theorem \ref{thm:intro} we will now show that for the class of polarized Weinstein domains admitting a Lyapunov function with only two critical points it is possible to apply the h-principle for the simplification of caustics directly, following the approach of \cite{St18}, and thus avoiding the more complicated treatment of \cite{AGEN20b}, which is only necessary when one needs to control the interaction of three or more strata in the skeleton. Moreover, for this special class of Weinstein domains we show that the skeleton can be arranged to have arboreal singularities of a particularly simple type, which does not directly follow from \cite{AGEN20b}. 

\begin{remark}
We claim no originality in the arborealization strategy, which simply follows the blueprint of \cite{St18}, all we do is verify that the necessary homotopical hypotheses are satisfied, which is a consequence of Theorem \ref{thm: main}.
\end{remark}

We recall that arboreal singularities are classified by finite rooted trees equipped with a decoration of signs $\pm 1$ on each edge not adjacent to the root. The height of a vertex is defined to be the minimal number of edges in a path between that vertex and the root. The height of a tree is defined to be the maximal height among all vertices. The height of an arboreal singularity is defined to be the height of the corresponding signed rooted tree.

\begin{corollary}
	Let $(W, \lambda, \phi)$ be a Weinstein manifold such that $TW$ admits a global field of Lagrangian planes and such that the Morse Lyapunov function $\phi$ only has two critical points. Then by a homotopy of the Weinstein structure we can arrange it so that the skeleton of $(W, \lambda)$ has arboreal, and moreover so that the arboreal singularities which appear in the skeleton have height $ \leq 2$.
\end{corollary}

\begin{proof}
	One of the critical points of $\phi$ is a minimum $x_0$. Let us assume that the other critical point $x_n$ has the maximal index $n$, the subcritical case being easier and left as an exercise for the reader. A neighborhood of $x_0$ is exact symplectomorphic to the standard Darboux ball $(B^{2n}, pdq-qdp)$ and the stable manifold $W^s(x_n)$ of $x_n$ intersects $\p B^{2n}=S^{2n-1}$ in an $(n-1)$-dimensional sphere $\Lambda$ which is Legendrian for the standard contact structure. 
	
	Following Starkston \cite{St18} we may deform the Weinstein structure of $(W, \lambda, \phi)$ in a neighborhood of $x_0$ from the standard Darboux model $(B^{2n}, pdq -qdp)$ to the standard cotangent model $(U^*D^n , pdq)$, where $U^*D^n = \{ (q,p) \in T^*D^n , \, \|p\| \leq 1 \}$. Moreover, we may arrange it so that the global Lagrangian distribution $\eta \subset TW$ agrees with the vertical distribution of $U^*D^n$ at its center point $x_0$, and hence after a homotopy of $\eta$ we may assume that it agrees with the vertical distribution on all of $U^*D^n$.
	
	Next, observe that $\p D^n \subset \p (U^*D^n)$ corresponds to a Legendrian unknot in $(\p B^{2n} , \xi_{\text{std} })$ which by a general position argument may be assumed to be disjoint from $\Lambda$. Hence we may now think of $\Lambda$ as a Legendrian submanifold in $S^*D^n = \{ (p,q) \in T^*D^n , \, \|p \|=1 \} \subset \p(U^*D^n)$. The singularities of the restriction $\pi|_\Lambda: \Lambda \to D^n$ of the front projection $\pi : S^*D^n \to D^n$ are the same as the tangencies of $\Lambda$ with respect to the distribution $\nu \subset T(S^*D^n)$ tangent to the fibres of $\pi$. Theorem \ref{thm:intro} says that it will be possible to deform $\Lambda$ by a Legendrian isotopy so that these singularities consist only of semi-cubical cusps as soon as we know that $\nu|_{\Lambda}$ is stably trivial as an element of $\pi_{n-1} \Lambda_{n-1}$. We remind the reader that here we are implicitly using the trivialization induced by a Weinstein neighborhood and the isomorphism of symplectic vector bundles $T(T^*\Lambda)|_\Lambda \simeq \Lambda \times \C^{n-1}$.
	
	Now, the image of $\nu_{\Lambda}$ under the stabilization map $\pi_{n-1} \Lambda_{n-1} \to \pi_{n-1} \Lambda_n$ can be identified with the direct sum of $\nu|_{\Lambda}$ with the Liouville direction, which is the vertical distribution of $U^*D^n$ restricted to $\Lambda$. By construction, on $U^*D^n$ this vertical distribution agrees with our globally defined Lagrangian field $\eta \subset TW$. In particular we see that the stabilization of $\nu|_\Lambda$ extends to the $n$-disk given by the stable manifold $W^s(x_n)$, which implies that the stabilization of $\nu|_{\Lambda}$ is trivial as an element of $\pi_{n-1} \Lambda_n$. This is precisely what we needed to show.
	
	Therefore by Theorem \ref{thm:intro} we may find a Legendrian isotopy $\Lambda_t$ of $\Lambda=\Lambda_0$ in $S^*D^n$ such that $\pi|_{\Lambda_1} : \Lambda_1 \to D^n$ has singularities consisting only of semi-cubical cusps, and this Legendrian isotopy can be realized by a homotopy of the ambient Weinstein structure. The new skeleton is arboreal outside of the cusp locus, with arboreal singularities of height $\leq 1$. To conclude the proof it remains to arborealize the semi-cubical cusps. To do this one may directly invoke \cite{St18}, hence the proof is complete.
	
	For the benefit of the reader let us briefly explain how this works. First, one introduces an explicit local model near the cusps to replace them with arboreal singularities, which are of height 2. The model propagates new arboreal singularities in the Liouville direction, so this modification is not local near the cusps. To fix this, one can insert a wall along $\Lambda_1$ on which the propagated singularities land. After a generic perturbation this results in new arboreal singularities of height 2 where before there were fold tangencies of $\Lambda_1$ with respect to $\nu$.  \end{proof}

\begin{figure}[h]
\includegraphics[scale=0.5]{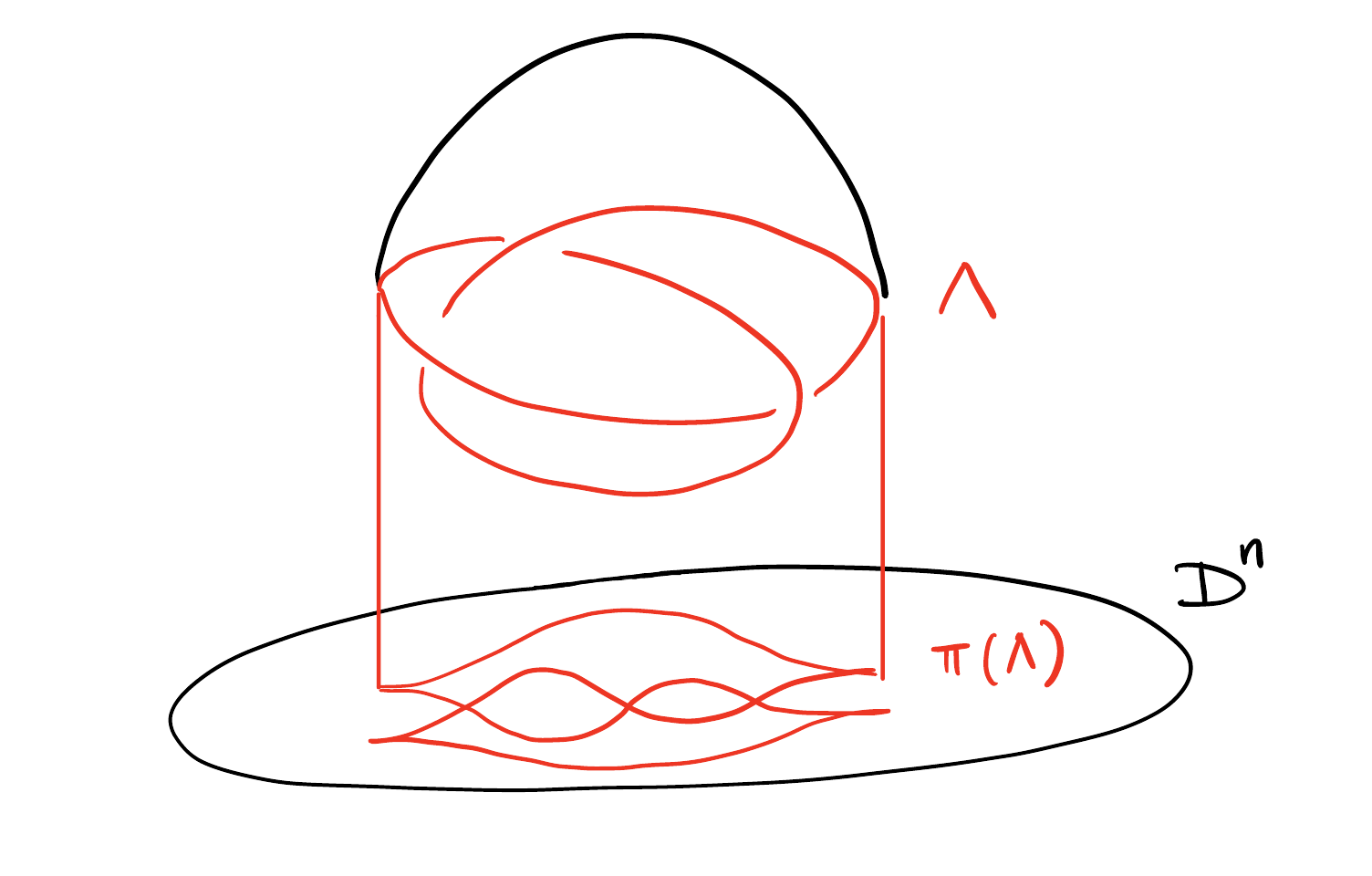}
\caption{The skeleton after the application of the h-principle but before trading the cusps for arboreal singularities of height 2.}
\label{arboreal}
\end{figure}

\subsection{Nearby Lagrangian homotopy spheres admit framed generating functions}\label{sec: ap2}

Let $\Sigma_0$ be an $n$-dimensional homotopy sphere and let $\Sigma_1 \subset T^*\Sigma_0$ be a Lagrangian embedding of another homotopy sphere $\Sigma_1$. In \cite{ACGK20} it is proved that the stable Gauss map $\Sigma_1 \to U/O$ is trivial, which is equivalent to the statement that the vertical distribution of $T^*\Sigma_0$ is stably trivial as a Lagrangian distribution defined along $\Sigma_1$. Therefore, Theorem \ref{thm:intro} implies the following result.

\begin{corollary}\label{cor: folded nearby}
	There exists a compactly supported Hamiltonian isotopy $\varphi_t$ of $T^*\Sigma_0$ such that $\varphi_1(\Sigma_1) \subset T^*\Sigma_0$ only has fold tangencies with respect to the vertical distribution.
\end{corollary}

This result has the following consequence. In \cite{ACGK20}, the triviality of the stable Gauss map $\Sigma_1 \to U/O$ is deduced as a consequence of an existence theorem for generating functions. This theorem states that $\Sigma_1$ can be presented as the Cerf diagram of a function $f:W \to \R$, where $W \to \Sigma_0$ is a bundle of tubes in the sense of Waldhausen \cite{W82}. We briefly recall the relevant definitions, and refer the reader to \cite{ACGK20} for futher details.

Let $E$ be a $k$-dimensional linear subspace of $\R^N$. Consider the codimension zero submanifold $T_E \subset \R^{N+1}$ obtained by attaching to the half-space $\{x_{N+1} \leq 0 \}$ a standard $(N+1)$-dimensional index $k$ handle along the unit sphere of $E \subset \{x_{N+1}=0\}$. We call $T_E$ a {\it rigid tube}. We call a {\it tube} any codimension zero submanifold $T \subset \R^{N+1}$ which is the image of a rigid tube under a smooth isotopy fixed outside of a compact set. 

\begin{figure}[h]
\includegraphics[scale=0.55]{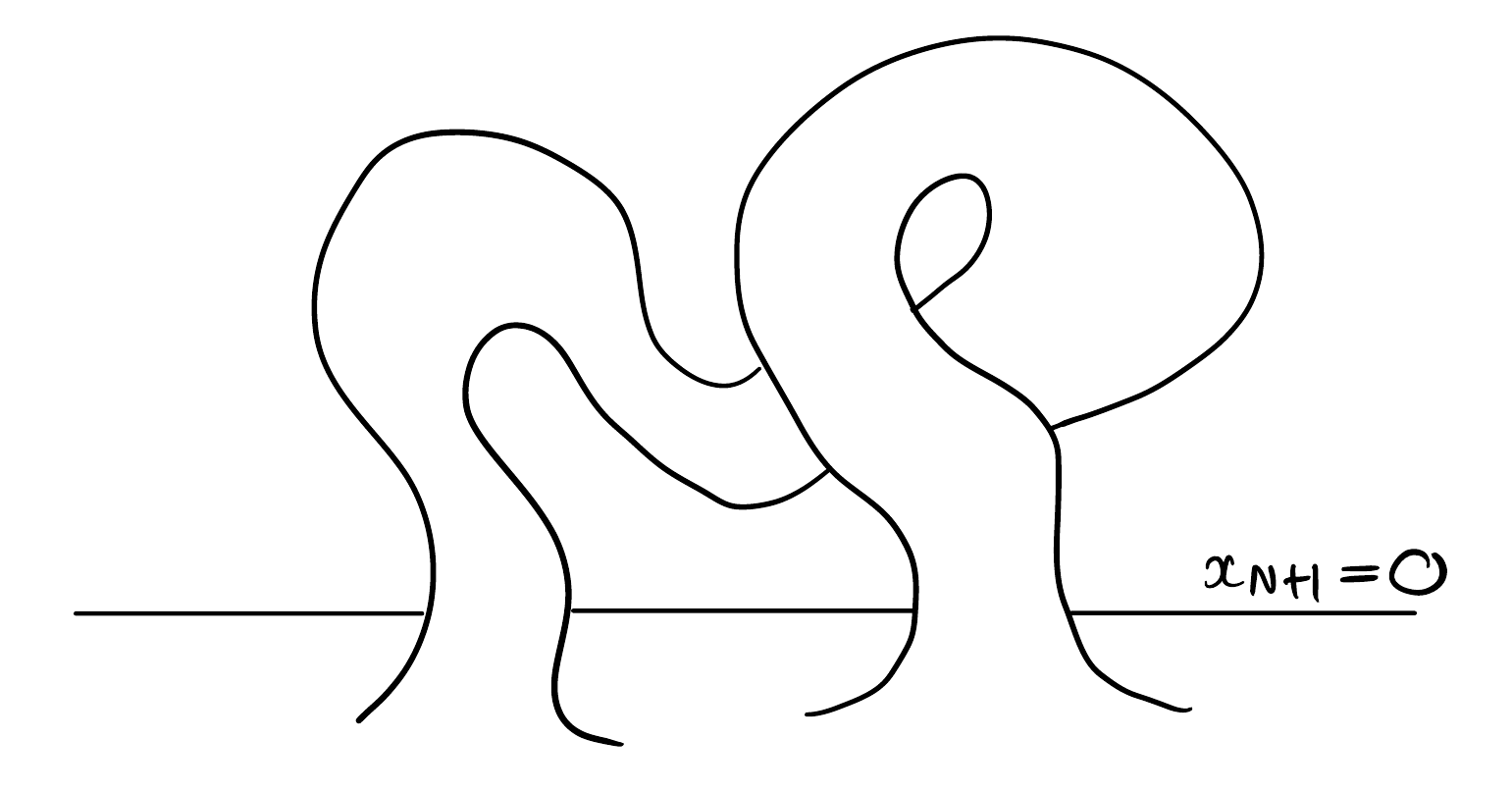}
\caption{A tube.}
\label{tube}
\end{figure}

\begin{definition} Let $M$ be a closed manifold. A {\em tube bundle} $W \to M$ is a smooth fibre bundle of manifolds whose fibres are tubes $T \subset \R^{N+1}$ in a fixed Euclidean space. \end{definition}

Let $T \subset \R^{N+1}$ be a tube. We consider functions $g:T \to \R$ such that:
\begin{itemize}
	\item[(1)] $\partial T$ is a regular level set of $g$.
	\item[(2)] $g=x_{N+1}$ outside of a compact set.
\end{itemize}

Let $W \to M$ be a tube bundle. We consider functions $f:W \to \R$ such that the restriction of $f$ to each fibre is a function satisfying (1) \& (2) and such that:
\begin{itemize}
	\item[(3)] the fibrewise Euclidean gradient $\nabla_Tf:W \to \R^{N+1}$ has 0 as a regular value.
\end{itemize}

We denote by $f_m:T_m \to \R$ the restriction of $f$ to the fibre over $m \in M$. 

\begin{definition} Let $W \to M$ be a tube bundle. The {\it Cerf diagram} of a function $f:W \to \R$ is the subset $\{(m,z) \in M \times \R : \, \, \text{$z$ is a crticial value of $f_m$} \} \subset M \times \R$. \end{definition}

Recall that $J^1M=T^*M \times \R$ is equipped with the canonical contact form $\alpha=dz-pdq$ for $pdq$ the Liouville form on $T^*M$. A Legendrian submanifold $\Lambda \subset J^1M$ is a smooth submanifold of the same dimension as $M$ such that $\alpha|_\Lambda=0$. The front projection of a Legendrian $\Lambda \subset J^1M$ is its image under the map $J^1M \to J^0M$, where we recall $J^0M=M \times \R$.

\begin{definition} Let $W \to M$ be a tube bundle, $\Lambda \subset J^1M$ a Legendrian submanifold and  $f:W \to \R$ a function satisfying (1), (2) \& (3). We say that $f$ is a {\em generating function} for $\Lambda$ if the symplectic reduction defines an embedding $\{\p_F f =0 \} \to J^1 M$ with image $\Lambda$.   \end{definition}

\begin{remark}
In particular, note the front projection of $\Lambda$ is  the Cerf diagram of $f$. 
\end{remark}

We can now state the existence theorem for generating functions:

\begin{theorem}[\cite{ACGK20}]\label{thm:existence-gf}
	Let $\Sigma_1 \subset T^*\Sigma_0$ be a Lagrangian homotopy sphere. Then there exists a tube bundle $W \to \Sigma_1$ and a function $f: W \to \R$ such that $f$ generates $\Lambda_1$.
\end{theorem}

Work in progress of the first author with K. Igusa aims to study such Lagrangian homotopy spheres via the the parametrized Morse theory of the generating function $f:W \to \R$, thought of as a family of functions on the fibres $f_x:W_x \to \R$, $x \in \Sigma_0$. However, theorem \cite{ACGK20} provides no a priori control over the singularities of this family. In particular, there is no guarantee that each $f_x$ is Morse or generalized Morse, nor can this be arranged by a generic perturbation. Here by generalized Morse we mean cubic, i.e. the normal form for Morse birth/death, so that at the moment of bifrucation $f_x$ takes the normal form given by $x^3$ direct sum the Morse normal form in the other coordinates. 

While existing h-principles in the literature \cite{I87}, \cite{EM12} ensure that the function $f:W \to \R$ may be deformed by a homotopy $f_t:W \to \R$ so that the restriction of $f_1:W \to \R$ to each fibre is Morse or generalized Morse, in general such a homotopy will generate a Lagrangian cobordism rather than a Lagrangian isotopy, and moreover will introduce self-intersection points so that in particular the end result is an immersed rather than embedded exact Lagrangian submanifold. One may overcome this issue by using Corollary \ref{cor: folded nearby} instead.

\begin{corollary}\label{cor: generalized}
	There exists a compactly supported Hamiltonian isotopy $\varphi_t$ of $T^*\Sigma_0$ such that $\varphi_1(\Sigma_1) \subset T^*\Sigma_0$ is generated by a function $f:W \to \R$ on a tube bundle $W \to \Sigma_0$ with the property that the restriction $f_x : W_x \to \R$ to each fibre is Morse or generalized Morse.
\end{corollary}

\begin{proof}
	This follows from a combination of Theorem \ref{thm:existence-gf}, which ensures existence of a generating function for $\Sigma_1$, together with Corollary \ref{cor: folded nearby} and the homotopy lifting property for generating functions. Indeed, if a function generates a Lagrangian whose tangencies with the vertical distribution consist only of folds, then the function has Morse birth/death singularities along the tangency locus and is Morse elsewhere. This is essentially the whole proof, but we have been vague about which version of the homotopy lifting property we are using so let us expand on this point.
	
	For example, if we use the version for fibrations at infinity as stated in \cite{EG98}, then the conclusion is that the isotopy $\varphi_t(\Lambda_1)$ may be generated by a homotopy of functions $f_t:W \times \R^{2m} \to \R$, where $f_0=f + Q(x,y)$ for $Q(x,y)=\|x\|^2-\|y\|^2$ and $f_t$ is a compact perturbation of $f$. This is not quite what we want, but the correction is a standard construction in pseudo-isotopy theory. Indeed, it is easy to get what we want by `folding down' the extra dimensions in the factor $\R^{2m}$, see for example \cite{I88}. We refer the interested reader to \cite{ACGK20}, where the appropriate notion of stabilization is discussed in detail. The result is a generating function on the $2m$-fold stabilization of $W$ in the sense of Waldhausen's stabilization map $\cT_{k,N} \to \cT_{k,N+1}$ \cite{W82}, which is still a tube bundle. This completes the proof.  \end{proof}
	
	\begin{figure}[h]
\includegraphics[scale=0.65]{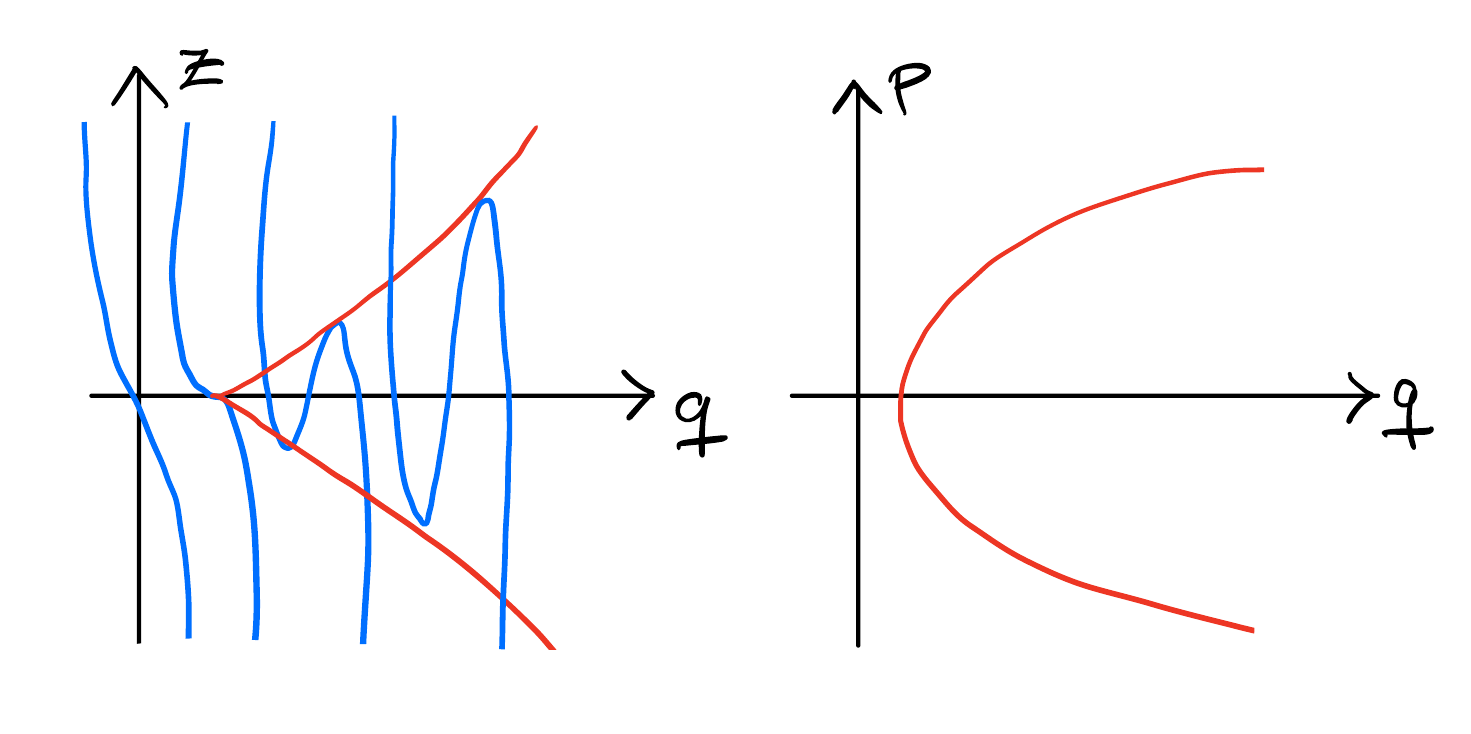}
\caption{A birth/death of Morse critical points corresponds to a semi-cubical cusp in the front projection and to a fold type tangency in the Lagrangian projection.}
\label{birthdeath}
\end{figure}

Finally, we prove that in the situation under consideration it is moreover possible to find a generating function $f:W \to \R$ which is {\em framed}, i.e. the restriction $f_x$ of $f$ to each fibre is Morse or generalized morse and furthermore the negative eigenspaces to the Hessian of $f_x$ at the critical points are equipped with framings that vary continuously with $x \in \Sigma_0$ and are suitably compatible at the birth/death points. Framed functions are useful because they are homotopically canonical \cite{I87}, \cite{EM12}, and can be used to compute higher K-theoretic invariants of the bundle they are defined over purely in terms of the associated family of Thom-Smale complexes \cite{I02}.

Intuitively, near a birth/death point the negative eigenspaces of the two critical points which come to be born or die differ by the 1-dimensional subspace in which the function is cubic, which is canonically framed by the direction in which the function is increasing. The compatibility requirement is that the framing for the negative eigenspace of the critical point of greater index is obtained from the framing of the negative space of the critical point of smaller index by adding the canonical framing of the cubic direction. 

An equivalent formulation (up to stabilization of $f$) is the following: for a function $f:W \to \R$ whose restriction to each fibre $f_x:W \to \R$ only has Morse or generalized Morse critical points, the negative eigenspaces to the Hessian of $f_x$ can be suitably stabilized depending on the index and assembled into a real vector bundle over the fibrewise critical locus of $W$, whose class in reduced topological K-theory is called the {\em stable bundle}. Then the condition that $f$ admits a framing is, up to stabilization, equivalent to the condition that the stable bundle is trivial, see \cite{I87} or \cite{EM12} for details.

\begin{corollary}\label{cor:framed}
	There exists a compactly supported Hamiltonian isotopy $\varphi_t$ of $T^*\Sigma_0$ such that $\varphi_1(\Sigma_1) \subset T^*\Sigma_0$ is generated by a framed function $f:W \to \R$ on a tube bundle $W \to \Sigma_0$.
\end{corollary}

\begin{proof}
	The function $f:W \to \R$ produced by Corollary \ref{cor:framed} need not admit a framing, however we may easily correct this. Let $E \to \varphi_1(\Sigma_1)$ be the stable bundle of $f$ as explained above. It is known that the projection $\Sigma_1 \to \Sigma_0$ is a homotopy equivalence \cite{A12}, hence $\varphi_1(\Sigma_1) \to \Sigma_0$ is also a homotopy equivalence, hence we may find a real vector bundle $F \to \Sigma_0$ such that the direct sum of $E$ with the pullback $\pi^*F$ of $F$ by $\varphi_1(\Sigma_1) \to \Sigma_0$ is trivial. 
	
	We may use $F$ to perform a twisted stabilization of $W$ to obtain a new tube bundle $W_F \to \Sigma_0$, namely this is the result of `folding down' the extra dimensions as before but this time with the function $f+Q_F$, where $Q_F: \R^{2m} \to \R$ is a family of quadratic forms parametrized by $\Sigma_0$ whose negative eigenspaces form a real vector bundle isomorphic to $F$. By construction, the new tube bundle $W_F$ has the property that there exists a function $g:W_F \to \R$ generating $\varphi_1(\Sigma_1)$ which near its critical points coincides with $f + Q_F$. Hence the new generating function $g$ has stable bundle $E \oplus \pi^*F$, which is trivial. This completes the proof. \end{proof}

 \end{document}